\documentclass[11pt,a4paper]{article}
\usepackage{amsmath, amsfonts, amssymb,amsthm}
\usepackage{a4wide}
\usepackage{parskip}
\usepackage{enumitem}
\usepackage{xcolor}
\usepackage{pstricks}
\usepackage{hyperref}
\usepackage{latexsym}
\usepackage{cite}
\usepackage{csquotes}
\usepackage{comment}
\usepackage{cite}
\def\QED{\hfill {$\square$}\goodbreak \medskip}
\usepackage[left=1in, right=1in,top=1in,bottom=1in]{geometry}
\allowdisplaybreaks
\newcommand{\Om} {\Omega}

\newcommand{\be} {\begin{equation}}
\newcommand{\ee} {\end{equation}}
\newcommand{\bea} {\begin{eqnarray}}
\newcommand{\eea} {\end{eqnarray}}
\newcommand{\Bea} {\begin{eqnarray*}}
\newcommand{\Eea} {\end{eqnarray*}}

\newcommand{\al} {\alpha}

\newcommand{\De} {\Delta}
\newcommand{\la} {\lambda}

\newcommand{\e} {\epsilon}
\newcommand{\2}{2^*_{\mu,s}}
\newcommand{\f}{\|u\|^2_{X_0}}
\def\R{{\mathbb R}}

\def\dx{\,{\rm d}x}
\def\dy{\,{\rm d}y}

\def\R{{\mathbb R}}

\newcommand{\ra} {\rightarrow}

\numberwithin{equation}{section}

\newtheorem{definition}{Definition}[section]
\newtheorem{theorem}{Theorem}[section]
\newtheorem{rem}{Remark}[section]
\newtheorem{lemma}{Lemma}[section]
\newtheorem{prop}{Proposition}[section]

\numberwithin{equation}{section}

\begin{document}
\setlength{\abovedisplayskip}{3pt}
\setlength{\belowdisplayskip}{3pt}
\date{}
	\title{Normalized solutions for fractional Choquard equation  with critical growth on  bounded domain}
    \author{ {\bf Divya Goel$\,^{1,}$\footnote{e-mail: {\tt divya.mat@iitbhu.ac.in}},  Asmita Rai$\,^{1,}$\footnote{e-mail: {\tt asmita.rai65@gmail.com}}} \\ $^1\,$Department of Mathematical Sciences, Indian Institute of Technology (BHU),\\ Varanasi 221005, India.}
		
	\maketitle
\begin{abstract}
In this work, we establish the multiplicity of positive solutions for the  following critical fractional Choquard equation with a perturbation on the star-shaped bounded domain 
\begin{equation*}    
 \begin{array}{cc}
(-\Delta)^s u = \lambda u +\alpha|u|^{p-2}u+ \left( \int\limits_{\Omega} \frac{|u(y)|^{2^{*}_{\mu ,s}}}{|x-y|^ \mu}\, \dy\right)  |u|^{2^{*}_{\mu ,s}-2}u\; \text{in} \; \Omega,\\
u>0\; \text{in}\; \Omega,\; \\ u = 0\; \text{in} \; \mathbb{R}^{N}\backslash\Omega, \\ \int_{\Omega}|u|^2\dx=d,
\end{array}
\end{equation*} 
where, $s\in(0,1), N>2s$, $\alpha\in \mathbb{R}$, $d>0$, $2<p<2^*_s:=\frac{2N}{N-2s}$ and $2^{*}_{\mu ,s}:=\frac{2N-\mu}{N-2s}$ represents fractional Hardy-Littlewood-Sobolev critical exponent.  Using the minimization technique over an appropriate set and the uniform mountain pass theorem, we prove the existence of first and second solutions, respectively.\\ 
\medskip
    
	\noindent \textbf{Key words:} Normalized solutions,  Fractional Laplacian, Choquard equation, Star-shaped bounded domain. 

\medskip

\noindent \textit{2020 Mathematics Subject Classification:} 35A15, 35J20, 35J60 
\end{abstract}
\maketitle
\section{Introduction}
This article is concerned with the existence of solutions for the fractional Choquard problem
\begin{align}    
\displaystyle  (-\Delta)^s u = \lambda u &+\alpha|u|^{p-2}u+ \left( \int\limits_{\Omega} \frac{|u(y)|^{2^{*}_{\mu ,s}}}{|x-y|^ \mu}\, \dy\right)  |u|^{2^{*}_{\mu ,s}-2}u, ~ \text{in} \; \Omega,~u = 0\; \text{in} \; \mathbb{R}^{N}\backslash\Omega, \label{b101}
\end{align} 
having prescribed mass $\|u\|_2^2=d$. Here $\Omega$ is a smooth bounded star-shaped domain of $\R^N$ with $N>2s$, $s\in(0,1),~\mu\in(0,N)$ and $d>0$. The parameter  $\alpha$ and $p$ satisfy one of the following conditions: 
\begin{itemize}
	\item[$(\mathcal{A}_1)$] $\alpha=0$,
    \item[$(\mathcal{A}_2)$] $\alpha>0,2+\frac{4s}{N}<p<\frac{2N}{N-2s}$,
	\item[$(\mathcal{A}_3)$] $\alpha<0, 2<p<\frac{2N}{N-2s}$. 
\end{itemize} 
 The operator  $(-\Delta)^s$ represents the fractional Laplacian, which is defined as 
\begin{equation*}
(-\Delta)^su(x)=-P.V.\int_{\R^N}\frac{u(x)-u(y)}{|x-y|^{N+2s}}\dy
\end{equation*}
  (up to a normalizing constant), using the Cauchy principal value (P.V.). This operator plays a significant role in the analysis of several physical phenomena, including crystal dislocation, minimal surfaces, phase transition, and Bose-Einstein condensation. Further details on these applications can be found in \cite{di2012hitchhikers, frohlich2007boson, longhi2015fractional}. 
  Laskin extended the classical local Laplacian by introducing fractional powers, thereby formulating a generalized approach to nonlinear Schrodinger equations. For more details, one can refer to \cite{applebaum2004levy,laskin2000fractional} and references therein. The problems involving the fractional Laplacian are of significant interest due to their rich mathematical structure. Over the last decade, mathematicians studied existence, multiplicity, bifurcation, and regularity results for these problems. For more details, one can refer \cite{bisci2016variational}.

The study of Choquard-type equations is extremely significant in many physical models. In 1954, Pekar \cite{pekar1954untersuchungen} discussed the following equation that arises in the quantum theory of polaron
\begin{equation*}
-\Delta u+u=\big(|x|^{-1} *|u|^2\big)u~~~\text{in}~\R^N.
\end{equation*}
Later, in 1976, by Choquard \cite{lieb1977existence}, the equation arose in modeling an electron trapped in its hole, under a certain approximation to the Hartree-Fock theory of one-component plasma.  Penrose\cite{penrose1996gravity} in his discussion on the self-gravitational collapse of quantum mechanical wave functional, and referred to it as the Schrodinger-Newton system.  Over the last decade, Choquard equations have received so much attention after the works of Moroz and Schaftingen \cite{moroz2013groundstates,moroz2015existence}  in the whole domain $\R^N$, and on bounded domains, Yang and Gao \cite{gao2018brezis}. 
For more works on the Choquard equation, we refer to \cite{liu2025sign,biswas2021variable,chen2024existence,alves2016existence}  and references therein with no intention to provide the full list.

From the mathematical approach, problem \eqref{b101} is doubly nonlocal due to the presence of terms $(-\Delta)^s$ and $\left( \displaystyle\int\limits_{\Omega} \frac{|u(y)|^{2^{*}_{\mu,s}}}{|x-y|^ \mu}\, \dy\right)$.  This causes some mathematical hurdles, which make the study of \eqref{b101} particularly interesting.
To solve \eqref{b101}, there are two different approaches.
The first way to solve \eqref{b101} is to consider $\la$ is fixed.  At present, we refer to this as a fixed frequency problem.  The existence, multiplicity and concentration of
solutions have already been studied for this type of problem. We refer the readers to \cite{d2015fractional,ma2017existence,yang2020multiplicity,mukherjee2016fractional}  and the references therein. The other point of view is to have a prescribed $L^2-$norm, i.e., $\|u\|_2^2=d$, for some $d>0$. When we fix the $L^2-$norm, $\lambda$ acts as a Lagrange multiplier. A solution with a prescribed $L^2-$norm is called a normalized solution. This type of solution is crucial in quantum physics, as wave functions ensure that the total probability of finding a particle equals one when $\|u\|_2=1$.\\
In these type of problems, we used the Gagliardo-Nirenberg inequality \cite{frank2016uniqueness}  to study the problem variationally. A new critical exponent $ p = 2 + 4s/N, ~ s \in (0,1)$ ($L^2$-critical exponent or mass critical exponent) appears for  \eqref{b101}. In general, we say that the region where $ r \in ( 2,  2 + 4s/N)$  is $L^2-$subcritical regime,  whereas $ r \in ( 2 +  4s/N, \infty)$   is $L^2-$supercritical regime. 

The study of normalized solutions starts with the work of Jeanjean \cite{jeanjean1997existence}. He studied the following problem ($s=1$)
\begin{equation}\label{b91}
\begin{cases}
&(-\Delta)^s u=\lambda u+g(u)~~\text{ in }\R^N,\\
&\|u\|_2^2=d, u \in H^s(\R^N).
\end{cases}
\end{equation}
He proved the existence of a normalized solution under the $L^2$-supercritical and Sobolev subcritical growth of the function $g$. Recently, Soave \cite{soave2020normalized1,soave2020normalized} proved the existence of ground state solution for the problem \eqref{b91} with $g(u)= |u|^{q-2} u+ |u|^{p-2}u$, $q\in(2,2^*)$ and $p\in (2+\frac{4}{N},2^*]$. Soave used the ideas of Jeanjean and the interplay between the $L^2$-subcritical and $L^2$-supercritical growth in the functional to establish the existence of a ground state solution.
After that, Luo and Zang \cite{luo2020normalized}  generalized the result of \cite{soave2020normalized} for $s \in (0,1)$. He, Radulescu, and Zou \cite{he2022normalized} extended this work on the problem of the fractional Laplacian when it has combined nonlinearity with the Choquard term. For more information, one can refer \cite{li2023nonexistence,zhang2022normalized,liu2019multiple,liu2024normalized,meng2024normalized,yu2023normalized} and references therein.

While normalized solutions have been extensively studied in the whole space $\R^N$,  there are very few works on a bounded domain.  Noris, Tavares, and Verzini\cite{noris2015existence} were the first to discuss the existence and orbital stability of ground states normalized solutions for $L^2-$critical and supercritical cases on a unit ball. Later, Verzini and Pierotti  \cite{pierotti2017normalized} established the existence of a solution using Morse theory for any bounded domain.  Recently, \cite{qi2024normalized},  Qi and Zou proved the existence of normalized solution in bounded domain for the following equation 
\begin{equation}\label{b2}
\begin{cases}
-\Delta u +&V(x)u =\lambda u+|u|^{p-2}u+|u|^{q-2}u ~~\text{ in }\Om\\
&\|u\|_2^2=d. 
\end{cases}
\end{equation}
Here, $2<p<2+\frac{4}{N}<q<2^*$. After that Lin, and  Lv \cite{lin2025normalized} studied \eqref{b2} for the supercritical nonlinearity. For more details, we refer  \cite{wang2024normalized,pierotti2025normalized,zhang2024normalized}. 

To the best of our knowledge,  there is no work which discusses normalized solutions to the Choquard equation over a bounded domain. Therefore, in this paper, we establish the existence and multiplicity of solutions of the following problem 
\begin{equation}  
\begin{cases}
  (-\Delta)^s u = \lambda u +\alpha|u|^{p-2}u+ \left( \int\limits_{\Omega} \frac{|u(y)|^{2^{*}_{\mu ,s}}}{|x-y|^ \mu}\, \dy\right)  |u|^{2^{*}_{\mu ,s}-2}u, ~ u>0 \; \text{in} \; \Omega,\\
 u = 0\; \text{in} \; \mathbb{R}^{N}\backslash\Omega, \\ \int_{\Omega}|u|^2\dx=d. 
\end{cases}
\tag{\(F_d\)}
\label{F}
\end{equation} 
Here, we prove the existence and multiplicity of the normalized solution. To establish the result, we used the variational methods. Since the domain is not scale invariant, we are not able to extend the techniques developed for the $\R^N$ case. To this, we wisely used the facts that $\Om$ is a star-shaped domain, every critical point of the associated function to \eqref{F} satisfies the Pohozaev identity and define a set $\mathcal{T}$ (see \eqref{T}). Subsequently,   we prove that the  associated energy functional is bounded below  on the $\mathcal{T}$.
Now, applying the minimization technique over $\mathcal{T}$, we prove the existence of a positive solution.  We even established that the energy of the solution is positive.  For the second solution,  we used the monotonicity technique of Jeanjean \cite{jeanjean1999existence}  to get the boundedness of the Palais-Smale sequence, and then we used the uniform mountain pass theorem. 

With this, we will define the weak solution corresponding to \eqref{F} and state our main results. To study the problem, we define the following space 
$$X_0:=\{u\in H^s(\R^N) : u=0~\text{a.e. in }\R^N\backslash\Omega\},$$
which forms a closed subspace of the fractional Sobolev space $H^s(\R^N)$ with the corresponding norm,
\begin{equation*}
 \|u\|_{X_0}=\left(\int_{\R^N}\int_{\R^N}\frac{|u(x)-u(y)|^2}{|x-y|^{N+2s}}\dx\dy\right)^{\frac{1}{2}}.  \end{equation*}
 \begin{definition} We say $u\in X_0$, is a weak solution of \eqref{F}, if it satisfy
    $$ \langle u,\phi\rangle_{X_0}=\lambda\langle u,\phi \rangle+\int_{\Omega}\int_{\Omega}\frac{|u(x)|^{\2}|u(y)|^{\2-1}\phi(y)}{|x-y|^{\mu}}\dx\dy+\alpha\int_{\Omega}|u(x)|^{p-1}\phi(x)\dx,$$
    for all $\phi\in X_0.$
\end{definition}

The  energy functional associated with \eqref{F} is defined as 
\begin{equation*}
{J}_d(u)=\frac{1}{2}\|u\|^2_{X_0}-\frac{\alpha}{p}\|u\|_p^p-\frac{1}{2.2^*_{\mu,s}}\|u\|_{NL}^{2\2},
\end{equation*}
under the constraint set $S_d:=\{u\in X_0:\|u\|_2^2=d\}.$\par
To find a positive solution of \eqref{F}, we examine the critical points of the following energy
\begin{equation*}
\tilde{J}_d(u)=\frac{1}{2}\|u\|^2_{X_0}-\frac{\alpha}{p}\|u^+\|_p^p-\frac{1}{2.2^*_{\mu,s}}\|u^+\|_{NL}^{2\2},
\end{equation*}
where $u^+=\text{max}\{u,0\}$ and 
$$ u \in S_d^+:=\{u\in X_0:\|u^+\|_2^2=d\}.$$
\begin{theorem}\label{B11}
	Let $\Om$ be a bounded domain with smooth boundary and assume that $\Om$ is star-shaped with respect to the origin. Suppose that the parameters $\alpha,p,d$ satisfy one of the following conditions:
	\begin{itemize}
		\item[(i)] $\alpha=0$,
        \begin{equation}\label{M1}
		d<\sup_{u\in S_1^+}\left(\min\left\{\left(\frac{\f}{\|u^+\|_{NL}^{2\2}  }\right)^{\frac{1}{\2-1}},\frac{2(\2-1)S_{HL}}{\2\|u\|_{X_0}^2} \right\}\right) ,
        \end{equation}
		\item[(ii)] $\alpha>0,2+\frac{4s}{N}<p<2^*_s$ and 
       
        \begin{equation}\label{M2}
\begin{aligned}
    d &< \sup_{u \in S_1^+} \Bigg(
    \min \Bigg(
    \max_{\xi \in (0,1)} \Bigg\{ 
        \Bigg( \xi \frac{\f}{\|u^+\|_{NL}^2} \Bigg)^{\frac{1}{\2 - 1}}, 
        \Bigg( (1 - \xi) \frac{\f}{\alpha \delta_p \|u^+\|_p^p} \Bigg)^{\frac{2}{p - 2}}
    \Bigg\}, \\
    &\quad \max_{\tau \in (0,1)} \Bigg\{ 
        \min \Bigg\{ 
        \Bigg( \frac{1}{2} - \frac{1}{p \delta_p} \Bigg) 
        \Bigg( \tau S_{HL}^{\2} \Bigg)^{\frac{2}{\2 - 2}} 
        \frac{2}{\|u\|_{X_0}^2}, \\
        &\qquad 
        \Bigg( \frac{1 - \tau}{\alpha \delta_p C_p^p} \Bigg)^{\frac{2}{p - 2}} 
        \Bigg( \Bigg( \frac{1}{2} - \frac{1}{p \delta_p} \Bigg) 
        \frac{2}{\|u\|_{X_0}^2} \Bigg)^{\frac{p \delta_p - 2}{p - 2}}
        \Bigg\}
    \Bigg\}
    \Bigg)
    \Bigg).
\end{aligned}
\end{equation}

        \item[(iii)] $\alpha<0,2<p<2^*_s$ and 
              \begin{align}\label{M3}
   \notag d &< \sup_{u \in S_1^+} \left( \min_{\xi \in (0,1)} \max \left\{ 
    \min \left\{ 
    \xi \frac{\2 - 1}{2\2} (S_{HL})^{\frac{\2}{\2 - 1}} \frac{2}{\f}, 
    \right. \right. \right. \\
    & \qquad \left. \left. \left. 
    \left( (1-\xi) \frac{\2 - 1}{2\2} (S_{HL})^{\frac{\2}{\2- 1}} \frac{p}{|\alpha| \|u^+\|_p^p} \right)^{\frac{2}{p}} 
    \right\} 
    \right\} \right).
\end{align}

	\end{itemize}
    Then the value 
    \begin{equation*}
        m_d:=\inf_{\mathcal{T}}\tilde{J}_d\in
        \begin{cases}
            \left(0,\frac{\2-1}{2\2}S_{HL}^{\frac{\2}{\2-1}}\right)~~~~\text{for}~\mathcal{A}_1,\mathcal{A}_3,\\
           \left(0, \left(\frac{1}{2}-\frac{1}{p\delta_p} \right)\max\limits_{\tau\in(0,1)}\min\left\{ \left(\tau S_{HL}^{\2} \right)^{\frac{2}{\2-2}},\left(\frac{1-\tau}{\alpha\delta_pC_p^pd^{p(1-\delta_p)/2}}\right)^{\frac{2}{p\delta_p-2}}\right \}\right)~~~\text{for}~\mathcal{A}_2,
        \end{cases}
    \end{equation*}
    and $m_d$ is achieved in $\mathcal{T}$, for some $u_d\in S_d^+$. Moreover, $u_d$ is a positive solution of \eqref{F} with Lagrange multiplier $\lambda_d$ satisfying $\lambda_d>0$ for $\mathcal{A}_1$, $\mathcal{A}_3$ and for $\mathcal{A}_2$, $\lambda_d<\lambda_{1,s}$, where $\lambda_{1,s}$ represents the first eigenvalue of $(-\De)^s$ on $\Om$ with Dirichlet boundary condition.
\end{theorem}

\begin{rem}
    $u_d$ is a local minimizer on $S_d^+$ and $\displaystyle\inf_{S_d^+}\tilde{J}_d=-\infty$ under our construction. Therefore, we can construct a mountain pass structure on $S_d^+$ and discuss the second solution for \eqref{F}. We define
$$\beta(d):=\inf_{\gamma\in \Gamma}\sup_{t\in[0,1]}\tilde{J}_d(\gamma(t))>\max\{\tilde{J}_d(u_d),\tilde{J}_d(v)\},$$
where $$\Gamma:=\{\gamma\in C([0,1],S_d^+):\gamma(0)=u_d,\gamma(1)=v\}.$$
\end{rem} 
\begin{theorem}\label{BB}
	Consider $\Om$ be a bounded, smooth and star-shaped domain with respect to the origin and suppose $\alpha,d,p$ satisfy the conditions as in Theorem \ref{B10} and also let $2<p<2^*-1$. Then \eqref{F} has positive solution $\hat{u}_d\ne u_d$ at the $\beta(d)$, and $\hat{u}_d$ is is of M-P type.
\end{theorem}
\textbf{Structure of the paper}: The Layout of this paper is  as follows:
\begin{itemize}
    \item In Section \ref{Sa1}, we provide some well-known results and variational framework.
    \item Section \ref{Sa2}, contains proof of Theorem
    \ref{B11}.
    \item In Section \ref{Sa3}, we establish a bounded P.S. sequence at the mountain pass energy level.
    \item Finally, Section \ref{Sa4} deals with the proof of Theorem \ref{BB}.
\end{itemize}
\maketitle
\section{Varitational framework}\label{Sa1}
In this section, we will state some important inequalities and discuss the framework that is required to prove our results:
     \begin{prop}[Hardy-Littlewood-Sobolev inequality]\label{hls}[Theorem 4.3,\cite{lieb2001analysis}]
	Let $t,r>1$ and $0<\mu<N$ with $\frac{1}{t}+\frac{\mu}{N}+\frac{1}{r}=2,f\in L^t(\R^N)$ and $h\in L^r(\R^N)$. There exists a sharp constant $C(t,r,\mu,N)$ independent of $f,h,$ such that 
	\begin{equation}\label{b1}
	\int_{\R^N}\int_{\R^N} \frac{f(x)h(y)}{|x-y|^{\mu}}\dx\dy \le C(t,r,\mu,N)\|f\|_{L^t}\|h\|_{L^r}.
	\end{equation}
	If $t=r=\frac{2N}{2N-\mu}$, then 
	$$C(t,r,\mu,N)=C(N,\mu)=\pi^{\frac{\mu}{2}}\frac{\Gamma(\frac{N}{2}-\frac{\mu}{2})}{\Gamma (N-\frac{\mu}{2})}\Big\{\frac{\Gamma(\frac{N}{2})}{\Gamma(\frac{\mu}{2})}\Big\}^{-1+\frac{\mu}{N}}.$$
	Equality holds in \eqref{b1} if and only if $f\equiv (constant)h$ and 
	$$h(x)=A(\gamma^2+|x-a|^2)^{\frac{2N-\mu}{2}},$$
	for some $A\in \mathbb{C}, 0\ne \gamma\in \R$ and $a\in \R^N$.
\end{prop}
From the embedding results in \cite{bisci2016variational}, the space $X_0$ is continuously embedded  into $L^r(\R)$ with $ r\in [1,2^*_s]$ where $2^*_s= \frac{2N}{N-2s}$. 
The best constant $S$ is defined  
\begin{align*}
S= \displaystyle\inf_{u \in X_0\setminus \{ 0\}}   \frac{\displaystyle\int_{\R}\int_{\R} \frac{|u(x)-u(y)|^2}{|x-y|^{N+2s}}~\dx\dy }{\left(\displaystyle \int_{ \Om} |u(x)|^{2^*}\dx\right)^{2/2^*_s}}. 
\end{align*}
Consequently, we define
\begin{align*}
S_{HL}= \inf_{u \in X_0\setminus \{ 0\}}   \frac{\displaystyle\int_{\R^N}\int_{\R^N} \frac{|u(x)-u(y)|^2}{|x-y|^{N+2s}}~\dx\dy }{\left( \displaystyle\int_{\R^N}\int_{\R^N} \frac{|u|^{\2}(x)|u|^{\2}(y)}{|x-y|^\mu}  ~\dx\dy \right)^{\frac{1}{\2}}}. 
\end{align*}

\begin{lemma}\label{lulem13}
	\cite{mukherjee2016fractional}
	The constant $S_{HL}$ is achieved if  and only if 
	\begin{align*}
	u=C\left(\frac{b}{b^2+|x-a|^2}\right)^{\frac{N-2s}{2}}
	\end{align*} 
	where $C>0$ is a fixed constant , $a\in \mathbb{R}^N$ and $b\in (0,\infty)$ are parameters. Moreover,
	\begin{align*}
	S=	S_{HL} \left(C(N,\mu)\right)^{\frac{N-2s}{2N -\mu}}.
	\end{align*}
\end{lemma}
Consider the family of minimizers $\{U_\e\}_{\e>0}$ of $S$ defined as 
\begin{align}\label{b35}
U_\e = \e^{-\frac{(N-2s)}{2}} S^{\frac{(N-\mu)(2s-N)}{4s(N-\mu+2s)}}(C(N,\mu))^{\frac{2s-N}{2(N-\mu+2s)}} u^*(x/\e)
\end{align}
where $u^*(x)= \overline{u}(x/S^{1/2s}),\; \overline{u} (x)= \frac{\tilde{u}(x)}{|\tilde{u}|_{2^*_s}}$ and $\tilde{u}(x)= a(b^2+|x|^2)^{\frac{-(N-2s)}{2}}$ with $ a \in \mathbb{R}\setminus\{0\}$ and  $b >0$ are fixed constants. Then  from Lemma \ref{lulem13}, for $\e>0, ~U_\e$ satisfies 
\begin{align*}
(-\De)^s u = (|x|^{-\mu}* |u|^{\2})|u|^{\2-2}u \text{ in } \R^N. 
\end{align*}
	 \begin{lemma}[Gagliardo-Nirenberg inequality]
	 	\cite{frank2016uniqueness} Let $p\in(2,2^*_s).$ Then, there exists a positive constant $C_p$ such that for all $u\in X_0$ 
        \begin{equation}\label{b31}
\|u\|_p^p \le C_p\|u\|_{X_0}^{p\delta_p}\|u\|_2^{p(1-\delta_p)}
\end{equation}
where $$\delta_p:=\frac{N(p-2)}{2ps}.$$	 	
	 \end{lemma}
     \begin{lemma}\label{B10}
         Let $u\in S_d^+$ be a critical point of $\tilde{J}_d$. Then, $u$ satisfies the following
         \begin{equation*}
	\|u\|^2_{X_0}-\frac{\Gamma(1+s)^2}{2}\int_{\Omega}\left(\frac{\partial u}{(\delta(x))^s}\right)^2(x.\rho)\rm d\sigma=\alpha\delta_p\|u^+\|_p^p+\|u^+\|_{NL}^{2\2},
	\end{equation*}
    where $\delta(x):=$ dist$(x,\partial\Omega)$, and $\rho$ denotes the exterior unit normal to $\partial\Omega$ at $x$.
     \end{lemma}
     \begin{proof}
        We have the desired result using \cite{mukherjee2016fractional, ros2014pohozaev} and that $u\in S_d^+$ is a critical point of $\tilde{J}_d$.
        \QED
    \end{proof}
    Observe that $\Omega$ is a star-shaped domain, so $x.\rho>0.$ It shows every critical point of $\tilde{J}_d$ will belong to $\mathcal{T}$ where 
	\begin{equation}\label{T}
	\mathcal{T}:=\{u\in S_d^+: \|u\|^2_{X_0}>\|u^+\|_{NL}^{2\2} +\alpha\delta_p \|u^+\|_p^p  \}.
    \end{equation}
	\begin{lemma}
		$\tilde{J}_d$ is coercive and bounded below on $\mathcal{T}$.
	\end{lemma}
\begin{proof}
	For any $u\in \mathcal{T}$ and $(\mathcal{A}_1)$ or $(\mathcal{A}_3)$ holds then, we have
	\begin{align}\label{b3}
	\notag \tilde{J}_d(u)&>\frac{\2-1}{2\2}\|u\|^2_{X_0}+\alpha\left(\frac{\delta_p}{2\2}-\frac{1}{p}\right)\|u^+\|_p^p\\
	&\ge \frac{\2-1}{2\2}\|u\|_{X_0}^2.
	\end{align} 
 Similarly, for $(\mathcal{A}_2)$, we have
  \begin{equation}\label{b4}
  \tilde{J}_d(u)> \left(\frac{1}{2}-\frac{1}{p\delta_p}\right)\|u\|_{X_0}^2.
  \end{equation}
  This implies $\tilde{J}_d$ is bounded from below and coercive on $\mathcal{T}$.\QED 
  \end{proof}
  \begin{rem}
       Define
  $$\partial\mathcal{T}:=\{u\in S_d^+: \|u\|_{X_0}^2= \|u^+\|_{NL}^{2\2}+\alpha\delta_p \|u^+\|_p^p\}. $$
  Then, using \eqref{b31}, under the assumption of $(\mathcal{A}_1)$ or
$(\mathcal{A}_3)$, we obtain that for any $u\in \partial \mathcal{T}$,
  \begin{align}\label{b5}
  \|u\|^2_{X_0}= \|u^+\|_{NL}^{2\2}+\alpha\delta_p \|u^+\|_p^p
  \le\|u^+\|^{2\2}_{NL}\le S_{HL}^{-\2}\|u\|_{X_0}^{2\2}.
  \end{align}
  Under the assumption $(\mathcal{A}_2)$, for $u\in\partial \mathcal{T}$
  \begin{align}\label{b6}
   \|u\|^2_{X_0}&= \|u^+\|^{2\2}_{NL}+\alpha\delta_p \|u^+\|^p_p\le S_{HL}^{\2}\|u\|_{X_0}^{\2}+\alpha\delta_p C^p_pd^{p(1-\delta_p)}\|u\|_{X_0} ^{p\delta_p}.
  \end{align}
  Hence, $\|u\|^2_{X_0}$ has a positive lower bound on $\partial\mathcal{T}$. Also, note that the lower bound depends on $L^2-$ norm of $u$, if $\alpha>0$.
  \end{rem}

\begin{prop}\label{B2}
	Let  $u$ be a critical point of $\tilde{J}_d|_{S_d^+}$, then the following holds
	\begin{itemize}
		\item[(i)] Any sequence $\{u_n\} \subset \mathcal{T}$ satisfying $\limsup\limits_{n\to\infty}J(u_n)<\infty$ is bounded in  $X_0$.
		\item[(ii)] If \eqref{M1}, \eqref{M2} or \eqref{M3} holds, then $\mathcal{T}\ne \emptyset$ and 
		\begin{equation}\label{b7}
		0<\inf_{u\in \mathcal{T}}\tilde{J}_d(u)<\inf_{u\in \partial\mathcal{T}} \tilde{J}_d(u).
		\end{equation}
		\item[(iii)] If $u\in \mathcal{T}$, there exists a unique $t_u>1$, such that $u^{t_u}\in \partial\mathcal{T},$ where  $u^t:=t^{\frac{N}{2}}u(tx),~ t\ge 1$ (as $\Omega$ is bounded domain, we can not allow $t<1$).  
	\end{itemize} 
\end{prop}
\begin{proof} Using Lemma \ref{B10} and the fact that $\Omega$ is star-shaped with respect to origin, we have $u\in \mathcal{T}$.\par 
	\textit{(i)} Let $\{u_n\}\subset \mathcal{T}$ be a sequence such that it satisfies $\limsup\limits_{n\to\infty}\tilde{J}_d(u_n)<\infty$. If $(\mathcal{A}_1)$ or $(\mathcal{A}_3)$ occurs then, by \eqref{b3} we have
	$$\limsup_{n\to\infty}\|u_n\|^2_{X_0}\le \left(\frac{2\2}{\2-1}\right)\limsup_{n\to\infty}\tilde{J}_d(u_n)<\infty. $$
	 Similarly, if $(\mathcal{A}_2)$ holds then, 
 using \eqref{b4}, we obtain
	 $$\limsup_{n\to\infty}\f \le \frac{2p\delta_p}{p\delta_p-2}\limsup_{n\to\infty}\tilde{J}_d(u_n)<\infty.$$
	 Thus, the sequence $\{u_n\}$ is bounded in $X_0$.\\
	 \textit{(ii)} The proof of this part, we will divide into 3 parts based on $(\mathcal{A}_1),~(\mathcal{A}_2)$ and $(\mathcal{A}_3)$. 
     \textbf{Case 1:} If $(\mathcal{A}_1)$ holds\par
     Let $u\in S_1^+$ and $d<\left(\frac{\f}{\|u^+\|_{NL}^{2\2}}\right)^{\frac{1}{\2-1}}$ , we obtain that
	 \begin{align*}
	\|\sqrt{d}u\|_{X_0}^2&=d\f\ge d^{\2}\|u^+\|_{NL}^{2\2}=\|\sqrt{d}u^+\|_{NL}^2.
	 \end{align*}
	 Hence, $\sqrt{d}u\in \mathcal{T}$, i.e., $\mathcal{T}\ne\emptyset$. Moreover, using \eqref{b5}, we get
	 $$\inf_{u\in \partial\mathcal{T}}\f\ge S_{HL}^{\frac{\2}{\2-1}}.$$
	 By using \eqref{b3} we have
	 $$\inf_{u\in \partial\mathcal{T}}\tilde{J}_d(v)\ge\frac{\2-1}{2\2}\f\ge\frac{\2-1}{2\2}S_{H,L}^{\frac{\2}{\2-1}}.$$
     For any $u\in S_1^+ $ and $d<\left(\left(\frac{\2-1}{2\2}\right)S_{HL}^{\frac{\2}{\2-1}}\frac{2}{\|u\|_{X_0}^2}\right)$, we obtain
     \begin{equation*}
         \tilde{J}_d(\sqrt{d}u)<\frac{d}{2}\f\le \frac{\2-1}{2\2}S_{HL}^{\frac{\2}{\2-1}}.
     \end{equation*}
     
	 When \eqref{M1} holds, then for some $u\in S_1^+$ such that $\sqrt{d}u\in \mathcal{T}$ and 
	 $$\inf_{u\in\mathcal{T}}\tilde{J}_d(u)\le \tilde{J}_d(\sqrt{d}u)<\inf_{u\in \partial\mathcal{T}}\tilde{J}_d(u).$$
	 Furthermore, using \eqref{b3} we obtain
	 $$\inf_{u\in\mathcal{T}}\tilde{J}_d(u)\ge \left(\frac{\2-1}{2\2}\right)\inf_{u\in \mathcal{T}}\f\ge \frac{\2-1}{2\2}\lambda_{1,s}d>0.$$
	 Therefore, \eqref{b7} holds.\par
	\textbf{Case 2:} If ($\mathcal{A}_2$) hold\\
    Let $u\in S_1^+$ and some $\xi \in (0,1),$ with  $d<\left\{\left(\xi\frac{\f}{\|u^+\|_{NL}^2 }\right)^{\frac{1}{\2-1}}, \left((1-\xi)\frac{\f}{\alpha\delta_p\|u^+\|_p^p} \right)^{\frac{2}{p-2}}\right\}$, implies that
	 \begin{align*}
	 \|\sqrt{d}u\|^2_{X_0}&=d\f\\
	 &=d\xi\f+d(1-\xi)\f\\
	 &>d^{\2}\|u^+\|_{NL}^2+\alpha\delta_pd^{\frac{p}{2}}\|u^+\|_p^p\\
	 &=\|\sqrt{d}u^+\|_{NL}^{2\2}+\alpha\delta_p\|\sqrt{d}u^+\|_p^p.
	 \end{align*}
	 Therefore, $\sqrt{d}u\in\mathcal{T}$, i.e., $\mathcal{T}\ne\emptyset$. By using \eqref{b6}, we deduce for any $\tau \in (0,1)$,\par
	 $$\inf_{u\in \partial\mathcal{T}}\|u\|_{X_0}\ge \left(\tau S_{HL}^{\2} \right)^{\frac{1}{\2-2}}$$
	 or
	 $$\inf_{u\in \partial\mathcal{T}}\|u\|_{X_0}\ge \left(\frac{1-\tau}{\alpha\delta_pC_p^pd^{p(1-\delta_p)/2}}\right)^{\frac{1}{p\delta_p-2}}.$$
	 From \eqref{b4}, for any $\tau \in (0,1)$, we have
	 $$\inf_{u\in \partial\mathcal{T}}\tilde{J}_d(u)\ge \left(\frac{1}{2}-\frac{1}{p\delta_p} \right)\min\left\{ \left(\tau S_{HL}^{\2} \right)^{\frac{2}{\2-2}},\left(\frac{1-\tau}{\alpha\delta_pC_p^pd^{p(1-\delta_p)/2}}\right)^{\frac{2}{p\delta_p-2}}\right \}.$$
	 For any $u\in S_1^+$ and
	 \begin{align*}
	 d<\max_{\tau\in(0,1)}&\left\{\min\left\{\left(\frac{1}{2}-\frac{1}{p\delta_p} \right)\left(\tau S_{HL}^{\2} \right)^{\frac{2}{\2-2}}\frac{2}{\|u\|_{X_0}^2},\right.\right. \\
     & \left.\left. \qquad \left(\frac{1-\tau}{\alpha\delta_pC_p^p} \right)^{\frac{2}{p-2}}\left(\left(\frac{1}{2}-\frac{1}{p\delta_p}\right)\frac{2}{\|u\|_{X_o}^2}\right)^{\frac{p\delta_p-2}{p-2}}\right\}\right\},
	 \end{align*}
	 there exists some $\tau_0\in (0,1)$ such that
	 \begin{align*}
	 \tilde{J}_d(\sqrt{d}u)&<\frac{d}{2}\f<  \left(\frac{1}{2}-\frac{1}{p\delta_p} \right)\min\left\{ \left(\tau_0 S_{HL}^{\2} \right)^{\frac{2}{\2-2}},\left(\frac{1-\tau_0}{\alpha\delta_pC_p^pd^{p(1-\delta_p)/2}}\right)^{\frac{2}{p\delta_p-2}}\right \}\le \inf_{u\in \partial\mathcal{T}}\tilde{J}_d(u).
	 \end{align*}
	When \eqref{M2} holds, we can choose some  $u\in S_1^+$ such that $\sqrt{d}u\in \mathcal{T}$ and
	 $$\inf_{u\in \mathcal{T}}\tilde{J}_d(u)\le \tilde{J}_d(\sqrt{d}u)<\inf_{u\in \partial\mathcal{T}}\tilde{J}_d(u).$$
	 Furthermore, by \eqref{b4} we obtain
	 $$\inf_{u\in \mathcal{T}}\tilde{J}_d(u)\ge \left(\frac{1}{2}-\frac{1}{p\delta_p} \right)\inf_{u\in \mathcal{T}}\f\ge \left(\frac{1}{2}-\frac{1}{p\delta_p} \right)\lambda_{1,s}d>0,$$
	 i.e. \textit{(ii)} holds.\par
     \textbf{Case 3:} If $(\mathcal{A}_3)$ holds\par
     Let $u\in S_1^+$ and $d<\left(\frac{\|u\|_{X_0}^2}{\|u^+\|_{NL}^{2\2}}\right)^{\frac{1}{\2-1}}$, we have
     \begin{align*}
         \|\sqrt{d}u\|_{X_0}^2&=d\|u\|_{X_0}^2     >d^{\2}\|u^+\|_{NL}^{2\2}
         >\|(\sqrt{d}u)^+\|_{NL}^{2\2}+\alpha\|(\sqrt{d}u)^+\|_p^p.
         \end{align*}
         Therefore, $\sqrt{d}u\in\mathcal{T}$, and so $\mathcal{T}\ne\emptyset.$ With the help of \eqref{b5}, we get
          $$\inf_{u\in \partial\mathcal{T}}\f\ge S_{HL}^{\frac{\2}{\2-1}}.$$
	 By using \eqref{b3} we have
	 $$\inf_{u\in \partial\mathcal{T}}\tilde{J}_d(v)\ge\frac{\2-1}{2\2}\f\ge\frac{\2-1}{2\2}S_{H,L}^{\frac{\2}{\2-1}}.$$
         For any $u\in S_1^+,$
         \begin{equation*}
             d<\max_{\xi\in(0,1)}\left\{\min\left\{\xi\frac{\2-1}{2\2}(S_{HL})^{\frac{\2}{\2-1}}\frac{2}{\f},\left((1-\xi)\frac{\2-1}{2\2}(S_{HL})^{\frac{\2}{\2-1}}\frac{p}{|\alpha|\|u^+\|_p^p} \right)^{\frac{2}{p}} \right\} \right\},
         \end{equation*}
        so, for some $\xi_0\in(0,1)$ such that
        \begin{align*}
            \tilde{J}_d(\sqrt{d}u)&<\frac{d}{2}\|u\|_{X_0}^2+\frac{|\alpha|}{p}d^{\frac{p}{2}} \|u^+\|_p^p\\
            &<\xi_0\frac{\2-1}{2\2}S_{H,L}^{\frac{\2}{\2-1}}+(1-\xi_0)\frac{\2-1}{2\2}S_{H,L}^{\frac{\2}{\2-1}}=\frac{\2-1}{2\2}S_{H,L}^{\frac{\2}{\2-1}}\\
            &<\inf_{u\in\partial\mathcal{T}}\tilde{J}_d(u).
        \end{align*}
        When \eqref{M3} holds, then for some $u\in S_1^+$, such that $\sqrt{d}u\in \mathcal{T}$ and 
	 $$\inf_{u\in\mathcal{T}}\tilde{J}_d(u)\le \tilde{J}_d(\sqrt{d}u)<\inf_{u\in \partial\mathcal{T}}\tilde{J}_d(u).$$
	 Furthermore, using \eqref{b3} we obtain
	 $$\inf_{u\in\mathcal{T}}\tilde{J}_d(u)\ge \left(\frac{\2-1}{2\2}\right)\inf_{u\in \mathcal{T}}\f\ge \frac{\2-1}{2\2}\lambda_{1,s}d>0.$$
     This completes the proof of \textit{(ii)}.\par
	 \textit{(iii)} Since $u^t(x)=t^{\frac{N}{2}}u(tx),t\ge1.$ For any $u\in S_d^+$, we have $u^t\in S_d^+$. Define $\Phi(t):=\tilde{J}_d(u^t)$, then 
	 \begin{equation*}
	 \Phi(t)=\frac{t^{2s}}{2}\|u\|^2_{X_0}-\frac{\alpha t^{sp\delta_p}}{p}\|u^+\|_p^p-\frac{t^{s2\2}}{2.2^*_{\mu,s}}\|u^+\|_{NL}^{2\2}.
	 \end{equation*}
	 It implies
	 \begin{align*}
	 \Phi'(t)&=st\f-s\alpha\delta_pt^{p\delta_p-1}\|u^+\|_p^p-st^{2\2-1}\|u^+\|^{22^*_{\mu,s}}_{NL}\\
	 &=st^{-1}\left(\|u^t\|^2_{X_0}     -\alpha\delta_p \|(u^t)^+\|_p^p-\|(u^t)^+\|_{NL}^{2\2} \right).
	 \end{align*}
	 Therefore, $u^t\in \mathcal{T}$ if and only if $\Phi'(t)>0$ and $u^t\in \partial\mathcal{T}$ if and only if $\Phi'(t)=0$. As $2^*_s>\max\{ 2,p\delta_p\}$, so we can obtain a unique $t_u>0$, such that, $\Phi'(t_u)=0$ and $\Phi'(t)>0$ in $(0,t_u)$ and $\Phi'(t)>0$ in $(t_u,\infty)$. Also, $u\in \mathcal{T}$ implies that $\Phi'(1)>0$, and therefore $t_u>1$. Further, $\Phi'(t_u)=0$ shows that $u^{t_u}\in \partial\mathcal{T}.$
\QED
\end{proof}

\maketitle
\section{Existence of a positive normalized solution}\label{Sa2}
In this section, we prove the existence of a positive solution of \eqref{F} using the minimization technique. \par
\textbf{Proof of Theorem \ref{B11} :} Assume $\{u_n\}\subset\mathcal{T}$ be a minimizing sequence of $m_d$,  then by Proposition \ref{B2} \textit{(i)}, $\{u_n\}$ sequence is bounded in $X_0$. Suppose, if possible, there exists $\{v_n\} \in \partial\mathcal{T}$ such that $u_n-v_n \to 0$ in $X_0$ up to a subsequence. As $\{u_n\}$ is bounded, so $\{v_n\}$ is bounded in $X_0$. So we obtain 
$$m_d=\lim_{n\to\infty}\tilde{J}_d(u_n)=\lim_{n\to\infty}\tilde{J}_d(v_n)\ge \inf_{v\in \partial\mathcal{T}}\tilde{J}_d(v) ,$$
which contradicts to \eqref{b7}. Therefore, $\{u_n\}$ is away from $\partial\mathcal{T}$.\par
By Ekeland's variational principle, we have that $\tilde{J}_d'|_{S_d^+}(u_n)=\tilde{J}_d'|_{\mathcal{T}}(u_n)\to 0$ as $n\to\infty$. Up to subsequences, assume that
\begin{align*}
u_n&\rightharpoonup u_d~~\text{weakly}~\text{in}~X_0,\\
u_n&\to u_d~~\text{strongly}~\text{in}~L^{p}(\Om),~\text{for}~2\le p<\2,\\
(|x|^{-\mu}*|u_n|^{\2})|u_n|^{\2}&\rightharpoonup (|x|^{-\mu}*|u_d|^{\2})|u_d|^{\2}~~\text{weakly}
~\text{in}~ L^{\frac{2N}{N+2s}}(\Om),\\
u_n&\to u_d~~\text{a.e. in }~\Omega.
\end{align*} 
Using trivial arguments, one can prove that $u_d\in S_d^+$ is a critical point of $\tilde{J}_d$ with constraints on $S_d^+$. Now employing  Proposition \ref{B2} \textit{(i)}, it follows that $u_d\in \mathcal{T}$ and thus $\tilde{J}_d(u_d)\ge m_d$.\par
On the other hand, let us assume that $w_n=(u_n-u_d)$. Since $(\tilde{J}_d|_{S_d^+})'(u_n)\to 0$ as $n\to\infty$, therefore by Lagrange multiplier rule, there exists $\lambda_n$ such that $\tilde{J}_d'(u_n)-\lambda_n u_n^+\to 0 $ as $n\to\infty$. If $\lambda_d$ is Lagrange multiplier corresponding to $u_d$, then for $\zeta \in X_0$ with $\int_{\Omega}u_d^+\zeta \dx\ne 0$, we have
\begin{equation*}
\lambda_n=\frac{1}{\int_{\Omega}u_n^+\zeta \dx}(\langle \tilde{J}_d'(u_n),\zeta\rangle+o_n(1))\to \frac{1}{\int_{\Omega}u_d^+\zeta \dx}\langle \tilde{J}_d'(u_d),\zeta\rangle=\lambda_d~\text{as}~n\to\infty.
\end{equation*}
By using the Brezis-Lieb Lemma \cite{brezis1983relation} and the fact that $\tilde{J}_d'(u_n)-\lambda_nu_n^+\to0,$ $\tilde{J}_d'(u_d)-\lambda_du_d^+=0,$
we obtain$$\|w_n\|^2_{X_0}=\|w_n\|_{NL}^{2\2}+o_n(1).$$
Let $\|w_n\|^2_{X_0}\to l\ge0$. Applying Brezis-Lieb Lemma, we have
\begin{align*}
\tilde{J}_d(w_n)=&\frac{1}{2}\|w_n\|^2_{X_0}-\frac{1}{2\2}\|w_n\|_{NL}^{2\2}+o_n(1)\\
=&\frac{\2-1}{2\2}l+o_n(1). 
\end{align*}
It implies $\tilde{J}_d(w_n)\ge o_n(1)$, and thus $m_d=\lim\limits_{n\to\infty}\tilde{J}_d(u_n)\ge \tilde{J}_d(u_d)$. Therefore, $\tilde{J}_d(u_d)=m_d$, and $l=0$ which implies $u_n\to u_d$ strongly in $X_0$. Subsequently, $u_d$ is a solution of the following equation 
\begin{equation}\label{b8}
(-\Delta)^su_d= \lambda u^+_d +\alpha|u^+_d|^{p-2}u_d+ \left( \int\limits_{\Omega} \frac{|u^+_d(y)|^{2^{*}_{\mu ,s}}}{|x-y|^ \mu}\, \dy\right)  |u_d^+|^{2^{*}_{\mu ,s}-2}u\; \text{in} \; \Omega,
\end{equation}
for some $\lambda_d\in \R$. Multiplying \eqref{b8} by $u_d^-$ and integrating on $\Omega$, we get $\|u^-_d\|_{X_0}=0$. Consequently, $u_d\ge 0$. By the strong maximum principle \cite{giacomoni2020regularity, cheng2017maximum}, we have $u_d>0$. Therefore, $\|u_d\|_2^2=\|u_d^+\|_2^2=d$, and $u_d$ solves \eqref{F}.
Now, we will prove that $\lambda_d>0$. For $\alpha\le0$, we obtain
\begin{align*}
\|u_d\|_{X_0}^2&=\lambda_dd+\alpha\|u_d\|_p^p+\|u_d\|^{22^*_{\mu,s}}_{NL}>\alpha\delta_p\|u_d\|^p_p+\|u_d\|^{22^*_{\mu,s}}_{NL},
\end{align*}
as $u_d\in\mathcal{T}$. Since $\delta_p<1$, this implies $\lambda_dd>\alpha(\delta_p-1)\|u_d\|_p^p.$ Therefore, $\lambda_d>0$. For $\al\geq 0 $, multiplying \eqref{F} by  $\phi_{1,s}$ ($\phi_{1,s}$ is the positive eigenvalue corresponding to $\lambda_{1,s}$) and  integrating it on $\Omega$, we obtain 
\begin{align*}
\lambda_d\int_{\Omega}u_d\phi_{1,s}\dx+\int_{\Omega}\int_{\Omega}\frac{|u(x)|^{\2}|u(y)|^{\2-2}u(y)\phi_{1,s}}{|x-y|^{\mu}}\dx\dy&+\alpha\int_{\Omega}|u_d|^{p-2}u_d\phi_{1,s}\dx\\&=\int_{\Omega}(-\Delta)^{\frac{s}{2}}u_d(-\Delta)^{\frac{s}{2}}\phi_{1,s}\dx\\
&=\lambda_{1,s}\int_{\Omega}u_d\phi_{1,s}\dx.
\end{align*}
It gives us $\lambda_d<\lambda_{1,s}$. Hence the proof follows.\QED
\maketitle
\section{Boundedness of P.S. sequence}\label{Sa3}
The aim of this section is to prove the boundedness of the Palais-Smale sequence of the functional $\tilde{J}_d$. To prove this, we define a family of functionals ($\tilde{J}_{d,\theta}$) and prove the boundedness of the Palais-Smale sequence of the

For  $\theta\in [1/2,1]$. Let
\begin{equation*}
 \tilde{J}_{d,\theta}=\left\{  \begin{array}{cc}
 \frac{1}{2}\|u\|_{X_0}^2-\theta\left(\frac{\alpha}{p}\|u^+\|_p^p+\frac{1}{2.2^*_{\mu,s}}\|u^+\|^{22^*_{\mu,s}}_{NL}\right)  &\text{if} \; \alpha>0,\\
\frac{1}{2}\|u\|_{X_0}^2-\frac{\alpha}{p}\|u^+\|^p_p-\frac{\theta}{2.2^*_{\mu,s}}\|u^+\|^{22^*_{\mu,s}}_{NL}&\text{if}~\alpha\le0. 
\end{array} \right.
\end{equation*}
 Subsequently we define 
$$\mathcal{T}_{\theta}:=\left\{u\in S_d^+: \|u\|_{X_0}^2\dx> \theta\left(\|u^+\|^{2\2}_{NL}+\alpha\delta_p \|u^+\|_p^p\right) \right\}, ~~\alpha>0$$
$$\mathcal{T}_{\theta}:=\left\{u\in S_d^+: \|u\|_{X_0}^2> \theta\|u^+\|^{2\2}_{NL}+\alpha\delta_p \|u^+\|_p^p   \right\},~~\alpha\le0. $$
Notice that \(\mathcal{T} \subset \mathcal{T}_{\theta}\) for \(\theta < 1\), which implies that \(\mathcal{T}_{\theta} \neq \emptyset\). Define
$$m_{d,\theta}:=\inf_{\mathcal{T}_{\theta}}\tilde{J}_{d,\theta}.$$
\begin{lemma}\label{B4}
	Let $d$ satisfy conditions \eqref{M1}, \eqref{M2}, \eqref{M3} depending on $\alpha$ and $p$. Then, we obtain the following results:
	\begin{itemize}
		\item[(i)]$\lim\limits_{\theta\to1^-}\inf\limits_{\partial\mathcal{T}}\tilde{J}_{d,\theta}=\inf\limits_{\partial\mathcal{T}}\tilde{J}_d,$
		\item[(ii)] $\liminf\limits_{\theta\to1^-}\inf\limits_{\mathcal{T}_{\theta}\backslash \mathcal{T}}\tilde{J}_{d,\theta}>\inf\limits_{\mathcal{T}}\tilde{J}_d,$
		\item[(iii)] $\lim\limits_{\theta\to1^-}m_{d,\theta}=m_d.$
	\end{itemize}
\end{lemma}
\begin{proof}\textit{(i)}
 By definition, $\tilde{J}_{d,\theta}(u)\ge \tilde{J}_d(u)$, for all $u\in \partial\mathcal{T}$. It implies
 \begin{equation*}	\inf\limits_{\partial\mathcal{T}}\tilde{J}_{d,\theta}\ge \inf\limits_{\partial\mathcal{T}}\tilde{J}_d,~~\forall~\theta\in[1/2,1].
\end{equation*}
	Thus, it is sufficient to show that $$\limsup\limits_{\theta\to1^-}\inf\limits_{\partial\mathcal{T}}\tilde{J}_{d,\theta}\le\inf\limits_{\partial\mathcal{T}}\tilde{J}_d.$$ 
    Let $\{u_n\}\subset \partial\mathcal{T}$ be a minimizing sequence, such that $\lim\limits_{n\to\infty}\tilde{J}_d(u_n)=\inf\limits_{\partial\mathcal{T}}\tilde{J}_d$. One can easily prove that $\{u_n\}$ is bounded in $X_0$. Therefore, we obtain
	\begin{align*}
\inf_{\partial\mathcal{T}}\tilde{J}_{d,\theta}\le\liminf\limits_{n\to\infty}\tilde{J}_{d,\theta}(u_n)=\lim\limits_{n\to\infty}\tilde{J}_d(u_n)+o(1)=\inf_{\partial\mathcal{T}}\tilde{J}_d+o(1),
	\end{align*}
	where $o(1)\to0$ as $\theta\to 1^-$. Hence, we have
    $$\limsup\limits_{\theta\to 1^-}\inf\limits_{\partial\mathcal{T}}\tilde{J}_{d,\theta}\le \inf\limits_{\partial\mathcal{T}}\tilde{J}_d.$$
	\par 
	\textit{(ii)} Let $\{u_n\}\subset \partial\mathcal{T}$ be a bounded minimizing sequence, such that $\lim\limits_{n\to \infty}\tilde{J}_d(u_n)=\inf\limits_{\partial\mathcal{T}}\tilde{J}_d$. Suppose $\{\theta_n\},~\{t_n\}$ are two sequences with $\theta_n\to1^-$ and  $t_n\to 1^+$, such that $u_n^{t_n}\in \mathcal{T}_{\theta_n}\backslash\mathcal{T}.$ Under these conditions, we derive 	$$\inf_{\mathcal{T}_{\theta_n}\backslash\mathcal{T}}\tilde{J}_d\le \tilde{J}_d(u_n^{t_n})=\tilde{J}_d(u_n)+o_n(1).$$   
	For $n\to\infty$, we have $\displaystyle\limsup\limits_{n\to\infty}\inf_{\mathcal{T}_{\theta_n}\backslash\mathcal{T}}\tilde{J}_d\le\inf_{\partial\mathcal{T}}\tilde{J}_d,$ which implies 
	\begin{equation}\label{b9}
	\limsup\limits_{\theta\to1^-} \inf\limits_{\mathcal{T}_{\theta_n}\backslash\mathcal{T}}\tilde{J}_d\le\inf\limits_{\partial\mathcal{T}}\tilde{J}_d.
	\end{equation}
	
	Moreover, for any $u\in \mathcal{T}_{\theta_n}\backslash\mathcal{T}$, it follows that
	$$\|u\|_{X_0}^2<\|u^+\|^{2\2}_{NL}+\alpha\delta_p \|u^+\|_p^p.$$
	Thus, under assumption $(\mathcal{A}_1)$ and $(\mathcal{A}_3)$, we can assert that
	$$\inf_{\mathcal{T}_{\theta_n}\backslash\mathcal{T}}\|u\|_{X_0}^2\ge S_{H,L}^{\frac{\2}{\2-1}}.$$
	In the similar manner, for the assumption $(\mathcal{A}_2)$, it can be established that for any  $\tau\in (0,1)$,
	
	$$\inf_{u\in\mathcal{T}_{\theta}\backslash \mathcal{T}}\f \ge \left(\tau S_{HL}^{\2} \right)^{\frac{1}{\2-2}}, $$
	or	$$\inf_{u\in\mathcal{T}_{\theta}\backslash \mathcal{T}}\f  \ge \left(\frac{1-\tau}{\alpha\delta_pC_ p^pd^{p(1-\delta_p)/2}}\right)^{\frac{2}{p\delta_p-2}}. $$

	Let $\{v_n\}$ be a minimizing sequence in $\mathcal{T}_{\theta}\backslash\mathcal{T}$, such that $\lim\limits_{n\to\infty}\tilde{J}_d(v_n)=\inf\limits_{\mathcal{T}_{\theta}\backslash \mathcal{T}}\tilde{J}_d$. For $(\mathcal{A}_1)$ and $(\mathcal{A}_3)$, we have	
	$$	\|v_n\|^2_{X_0}> \|v_n^+\|^{2\2}+\alpha\delta_p \|v_n^+\|_p^p.$$
	Thus, when $\theta$ approaches  $1^-$, such that $p\delta_p<2\2\theta$, we establish the relation
	\begin{align}\label{b10}
	\notag \tilde{J}_d(v_n)&>\left(\frac{1}{2}-\frac{1}{2\2\theta} \right)\|v_n|_{X_0}^2+\alpha\left(\frac{\delta_p}{2\2\theta}-\frac{1}{p}\right)\|v_n^+\|_p^p\\
	&\ge \left(\frac{1}{2}-\frac{1}{2\2\theta} \right)\|v_n\|_{X_0}^2.
	\end{align}
	For  assumption $(\mathcal{A}_2)$, we have
	$$\|v_n\|_{X_0}^2> \theta\left(\|v_n^+(y)|^{2\2}_{NL}+\alpha\delta_p \|v_n^+\|_p^p\right). $$
	Therefore, for $\theta$ near to $1^-$, such that $p\delta_p\theta>2$ we obtain
	\begin{equation}\label{b11}
	\tilde{J}_d(v_n)>\left(\frac{1}{2}-\frac{1}{p\delta_p\theta}\right)\|v_n\|_{X_0}^2.
	\end{equation}
Utilizing \eqref{b9},\eqref{b10} and \eqref{b11}, we conclude that the sequence $\{v_n\}$ is uniformly bounded in $X_0$ as $\theta$ approaches $1^-$. Given that $d$ satisfies any one of \eqref{M1}, \eqref{M2} or \eqref{M3}, we can easily prove the desired result using the same arguments as in the proof of Proposition~\ref{B2} \textit{(ii)}.\par 
\textit{(iii)} Using the fact that  $u_d\in \mathcal{T}\subset\mathcal{T}_{\theta}$, we obtain
\begin{equation*}
m_{d,\theta}\le \tilde{J}_{d,\theta}(u_d)=\tilde{J}_d(u_d)+o_{\theta}(1)=m_d+o_{\theta}(1). 
\end{equation*}
Applying limit $\theta\to1^-$, we get
\begin{equation*}
\limsup\limits_{n\to\infty} m_{d,\theta}\le m_d.
\end{equation*}
Let $\{\theta_n\}$ be any sequence, such that,  $\{\theta_n\}\ra 1^-$, and $\{w_n\}\subset \mathcal{T}_{\theta_n}$ be a sequence such that $\tilde{J}_{d,\theta_n}(w_n)\to m_{d,\theta_n}$ as $n\to\infty.$ By using  $\limsup\limits_{n\to\infty} m_{d,\theta_n}\le m_d$ and $\{w_n\}\subset \mathcal{T}_{\theta}$, we conclude that $\{w_n\}$ is bounded in $X_0$.
Moreover, previous analysis confirms that $\{w_n\}\subset\mathcal{T}$, for large $n$. Since
$$\limsup\limits_{n\to\infty} \tilde{J}_d(w_n)=\limsup\limits_{n\to\infty} \tilde{J}_{d,\theta_n}(w_n)=\limsup\limits_{n\to\infty} m_{d,\theta_n}\le m_d.$$
This leads to
$$\tilde{J}_{d,\theta_n}(w_n)=\tilde{J}_d(w_n)+o_n(1)\ge m_d+o_n(1).$$
Taking limits as $n\to\infty$, we derive
$$\liminf\limits_{n\to\infty} m_{d,\theta_n}=\liminf\limits_{n\to\infty} \tilde{J}_{d,\theta_n}(w_n)\ge m_d,$$
which implies
$$\liminf\limits_{\theta\to 1^-} m_{d,\theta}\ge m_d.$$
Hence, the result follows. \QED
\end{proof}
\begin{lemma}\label{B5}
	Let $d$ satisfy any one of the conditions \eqref{M1}, \eqref{M2}, \eqref{M3}, according to $\alpha$ and $p$, then
    \begin{itemize}
 \item[(i)] there exist $\epsilon\in (0,1/2)$ and $\delta>0$ independent of $\theta$, such that 
	\begin{equation}\label{b12}
	\tilde{J}_{d,\theta}(u_d)+\delta<\inf_{\partial\mathcal{T}}\tilde{J}_{d,\theta},~~~\forall \theta\in (1-\epsilon,1],
	\end{equation}
	and $v\in S_d^+\backslash\mathcal{T}$, such that
	$$\beta_{\theta}:=\inf_{\gamma\in \Gamma}\sup_{t\in[0,1]}\tilde{J}_{d,\theta}(\gamma(t))>\tilde{J}_{d,\theta}(u_d)+\delta=\max\{\tilde{J}_{d,\theta}(u_d),\tilde{J}_{d,\theta}(v) \}+\delta,$$
	where
	$$\Gamma:=\{\gamma\in C([0,1],S_d^+):\gamma(0)=u_d,\gamma(1)=v\},$$
	is independent of $\theta$.
   \item[(ii)]  $\lim\limits_{\theta\to 1^-}\beta_{\theta}=\beta_1$.
     
    \end{itemize}
\end{lemma}
\begin{proof}\textit{(i)}
  Using  Proposition \ref{B2} \textit{(ii)} and Lemma \ref{B4} \textit{(ii)} and the fact that  $\lim\limits_{\theta\to1^-}\tilde{J}_{d,\theta}(u_d)=\tilde{J}_d(u_d)=m_d$,
we deduce that 
$$\lim\limits_{\theta\to 1^-}\tilde{J}_{d,\theta}(u_d)=m_d<\inf_{\partial\mathcal{T}}\tilde{J}_d=\lim\limits_{\theta\to 1^-}\inf_{\partial\mathcal{T}}\tilde{J}_{d,\theta}.$$
Defining  $2\delta:=(\inf\limits_{\partial\mathcal{T}}\tilde{J}_d-m_d)$, and  taking $\epsilon$ small enough, we obtain \eqref{b12}. Let 
$u^t=t^{\frac{N}{2}}u(tx)\in S_d^+$ for $t\ge1$ and $u\in S_d^+$.  Using the fact that  $2\2>2^*_s>\max\{2,p\delta_p\}$  under the assumption  $(\mathcal{A}_1)$ and $(\mathcal{A}_3)$, we obtain
\begin{align*}
\tilde{J}_{d,\theta}(u^t)&=\frac{1}{2}\|u^t\|_{X_0}^2-\frac{\alpha}{p}\|(u^t)^+|_p^p-\frac{\theta}{2.2^*_{\mu,s}}\|(u^t)^+\|_{NL}^{2\2}\\
&\le\frac{t^{2s}}{2}\|u\|_{X_0}^2+\frac{|\alpha| t^{sp\delta_p}}{p}\|u^+\|_p^p-\frac{t^{s2\2}}{2.2^*_{\mu,s}}\|u^+\|^{22^*_{\mu,s}}_{NL}\\
&\to -\infty~\text{uniformly w.r.t}~\theta\in [\frac{1}{2},1], ~t\to\infty.
\end{align*}
Under the assumption  $(\mathcal{A}_2)$, we obtain  
\begin{align*}
\tilde{J}_{d,\theta}(u^t)&=\|u^t\|_{X_0}^2-\theta\left(\frac{\alpha}{p}\|(u^t)^+|_p^p+\frac{1}{2.2^*_{\mu,s}}\|(u^t)^+\|^{22^*_{\mu,s}}_{NL}\right)\\
&\le\frac{t^{2s}}{2}\|u\|_{X_0}^2-\frac{1}{2}\left(\frac{\alpha t^{sp\delta_p}}{p}\|u^+\|_p^p+\frac{t^{s2\2}}{2.2^*_{\mu,s}}\|u^+\|^{22^*_{\mu,s}}_{NL}\right)\\
&\to -\infty~\text{uniformly w.r.t}~\theta\in [\frac{1}{2},1], ~ t\to\infty. 
\end{align*}
Let  $v=u^t$  then  we can find  $t$ large enough such that $\tilde{J}_{d,\theta}(v)<\tilde{J}_{d,\theta}(u_d)$. Using a similar argument as in Proposition \ref{B2} \textit{(iii)}, we can assume $v\not\in\mathcal{T}$. Then for any $\gamma\in \Gamma$, there exists $\tilde{t}\in (0,1)$ such that $\gamma(\tilde{t})\in \partial\mathcal{T}$. Therefore,
$$\inf_{\gamma\in \Gamma}\sup_{t\in[0,1]}\tilde{J}_{d,\theta}(\gamma(t))\ge\inf_{u\in \partial\mathcal{T}}\tilde{J}_{d,\theta}(u).$$
\textit{(ii)} Let $u\in S_d^+$, then $\tilde{J}_{d,\theta}(u)\ge \tilde{J}_d(u)$ for $\theta<1.$ This implies that $\lim\limits_{\theta\to 1^-}\inf
\beta_{\theta}\ge\beta_1$. Now, we  claim that $\lim\limits_{\theta\to 1^-}\sup \beta_{\theta}\le\beta_1.$ Using the definition of $\beta_{1},  $ for any $\epsilon>0$ we can take a $\gamma_0\in \Gamma$ such that
$$\sup_{t\in[0,1]}\tilde{J}_d(\gamma_0(t))<\beta_{1}+\epsilon.$$
For any $\theta_n\to 1^-$, we have 
\begin{align*}
\beta_{\theta_n}&=\inf_{\gamma\in \Gamma}\sup_{t\in[0,1]}\tilde{J}_{d,\theta_n}(\gamma(t))\\&\le\sup_{t\in[0,1]}\tilde{J}_{d,\theta_n}(\gamma_0(t))\\&=\sup_{t\in[0,1]}\tilde{J}_d(\gamma_0(t))+o_n(1)\\&<\beta_{1}+\epsilon+o_n(1).
\end{align*}
By arbitrariness of $\epsilon$, we have
$$\lim\limits_{n\to\infty}\beta_{\theta_n}\le\beta_{1},$$
which implies
$$\lim\limits_{\theta\to1^-}\sup \beta_{\theta}\le\beta_{1}.$$\qed
\end{proof} 
\begin{prop}[Monotonicity trick\cite{jeanjean1999existence}]\label{B6}
	Let $I\subset \R^+$  be an interval. Consider a family $(\tilde{J}_{d,\theta})_{\theta\in I}$ of $C^1-$ functionals on $X_0$ of the form
	$$\tilde{J}_{d,\theta}(u)=A(u)-\theta B(u),~~\theta\in I$$
	where $B(u)\ge 0,\forall~ u\in X_0$ and such that $A(u)\to\infty$ or $B(u)\to\infty$ as $\|u\|_{X_0}\to\infty$. Suppose $w_1,w_2\in S_d^+$ (independent of $\theta$), such that 
	$$\Gamma=\{\gamma\in C([0,1],S_d^+), \gamma(0)=w_1,\gamma(1)=w_2\},$$
for all $\theta\in I$,
$$\beta_{\theta}:=\inf_{\gamma\in \Gamma}\sup_{t\in[0,1]}\tilde{J}_{d,\theta}(\gamma(t))>\max\{\tilde{J}_{d,\theta}(w_1),\tilde{J}_{d,\theta}(w_2)\}.$$
Then, for almost every $\theta\in I$, there exists a sequence $\{w_n\}\subset S_d^+$, such that
\begin{itemize}
	\item[(i)] $\{w_n\}$ is bounded in $X_0$,
	\item[(ii)] $\tilde{J}_{d,\theta}(w_n)\to\beta_{\theta},$
	\item[(iii)] $(\tilde{J}_{d,\theta})|_{S_d^+}(w_n)\to 0$ in $X_0'.$
\end{itemize} 
\end{prop}

\begin{prop}\label{B7}
	 Let $d$ satisfy any one of the conditions \eqref{M1},\eqref{M2},\eqref{M3} according to $\alpha$. Also, assume that $N\in(2s,6s)$ for conditions $\mathcal{A}_2$ and $\mathcal{A}_3$. Then, for almost every $\theta\in[1-\epsilon,1]$, there exists a critical point $u_{\theta}$ of $\tilde{J}_{d,\theta}$ constrained on $S_d^+$ at the level $\beta_{\theta},$ which solves
	\begin{equation}\label{b13}   
	\left\{  \begin{array}{cc}
	\displaystyle  (-\Delta)^s u_{\theta} = \lambda_{\theta} u_{\theta} +\alpha|u_{\theta}|^{p-2}u+ \theta\left( \int\limits_{\Omega} \frac{|u_{\theta}(y)|^{2^{*}_{\mu ,s}}}{|x-y|^ \mu}\, \dy\right)  |u_{\theta}|^{2^{*}_{\mu ,s}-2}u_{\theta}\; \text{in} \; \Omega,\\
	u_{\theta}>0\; \text{in}\; \Omega,\;  u_{\theta} = 0\; \text{in} \; \mathbb{R}^{N}\backslash\Omega, \\ 
	\end{array} \right.
	\end{equation} 
	for $\alpha\le0$, and
	\begin{equation}\label{b14}  
	\left\{  \begin{array}{cc}
	\displaystyle  (-\Delta)^s u_{\theta} = \lambda_{\theta} u_{\theta} +\alpha\theta|u_{\theta}|^{p-2}u+ \theta\left( \int\limits_{\Omega} \frac{|u_{\theta}(y)|^{2^{*}_{\mu ,s}}}{|x-y|^ \mu}\, \dy\right)  |u_{\theta}|^{2^{*}_{\mu ,s}-2}u_{\theta}\; \text{in} \; \Omega,\\
	u_{\theta}>0\; \text{in}\; \Omega,\;  u_{\theta} = 0\; \text{in} \; \mathbb{R}^{N}\backslash\Omega, \\ 
	\end{array} \right.
	\end{equation} 
for $\alpha>0$, and for some $\lambda_{\theta}$. Additionally, $u_{\theta}\in \mathcal{T}_{\theta}.$
	
\end{prop}
\begin{proof}
Let 
\begin{equation*}
A(u)=\left\{  \begin{array}{cc}
 \frac{1}{2}\|u\|_{X_0}^2  &\text{if} \; \alpha>0,\\
\frac{1}{2}\|u\|_{X_0}^2-\frac{\alpha}{p}\|u^+\|_p^p&\text{if}~\alpha\le0, 
\end{array} \right.
\end{equation*}
 \begin{equation*}
B(u)=\left\{  \begin{array}{cc}
 \|u^+\|^{2\2}_{NL}+\frac{\alpha}{p}\|u^+\|_p^p  &\text{if} \; \alpha>0,\\
\|u^+\|^{2\2}_{NL}&\text{if}~\alpha\le0. 
\end{array} \right.
\end{equation*}
Then using the Proposition \ref{B6},  for all  $\theta\in [1-\epsilon,1]$, there exists a bounded P.S. sequence $\{u_n\}\subset S_d^+$, satisfying $\tilde{J}_{d,\theta}(u_n)\to\beta_{\theta}$ and $(\tilde{J}_{d,\theta}|_{S_d^+})'(u_n)\to 0$ as $n\to\infty$. It implies that up to a 
 \begin{align*}
 u_n&\rightharpoonup u_{\theta}~~~\text{weakly in}~X_0,\\
 u_n&\to u_{\theta}~~\text{strongly}~\text{in}~L^{p}(\Om)~\text{for}~2\le p<\2,\\
(|x|^{-\mu}*|u_n|^{\2})|u_n|^{\2}&\rightharpoonup (|x|^{-\mu}*|u_{\theta}|^{\2})|u_{\theta}|^{\2}~~\text{weakly}
~\text{in}~ L^{\frac{2N}{N+2s}}(\Om),\\
 u_n&\to u_{\theta}~~~\text{a.e. in}~\Omega. 
  \end{align*}
 One can easily verify that  $u_{\theta}\in S_d^+$ is a critical point of $\tilde{J}_{d,\theta}{|_{S_d^+}}$. Define $w_n=u_n-u_{\theta}$. Since $(\tilde{J}_{d,\theta}|_{S_d^+})'(u_n)\to 0$, it follows that there exists $\lambda_n$ such that $(\tilde{J}_{d,\theta})'(u_n)-\lambda_nu_n^+\to0$. Let $\lambda_{\theta}$ be a Lagrange multiplier corresponding to $u_{\theta}.$ Using the same arguments as in  the theorem \ref{B11}, we have 
$$\|w_n\|_{X_0}^2=\theta\|w_n\|^{22^*_{\mu,s}}_{NL}+o_n(1).$$
Assuming that $\|w_n\|_{X_0}^2\to \theta l\ge0 $ and $\|w_n\|^{22^*_{\mu,s}}_{NL}\to l\ge0,$ we can use the definition of $S_{HL}$ to obtain
  $$ \|w_n\|_{X_0}^2\ge S_{HL}\left(\|w_n\|^{22^*_{\mu,s}}_{NL} \right)^{\frac{1}{\2}}.$$
  This implies that $\theta l>S_{HL}l^{\frac{1}{\2}}$. We claim that $l=0$. Suppose if possible $l>0$, then $l\ge \left(\frac{S_{HL}}{\theta}\right)^{\frac{\2}{\2-1}}$, consequently we get 
  $$\tilde{J}_{d,\theta}(w_n)\ge \left(\frac{\2-1}{\2} \right)\left(\frac{S_{HL}^{\2}}{\theta} \right) ^{\frac{1}{\2-1}}+o_n(1).$$
  Now using the Brezis-Lieb Lemma, we have
  \begin{align}\label{b40}
  \notag \tilde{J}_{d,\theta}(u_n)&= \tilde{J}_{d,\theta}(u_{\theta})+ \tilde{J}_{d,\theta}(w_n)+o_n(1)\\
  &\ge m_{d,\theta}+\left(\frac{\2-1}{\2} \right)\left(\frac{S_{HL}^{\2}}{\theta} \right) ^{\frac{1}{\2-1}}+o_n(1).
  \end{align}
  By means of  Proposition \ref{B6} \textit{(iii)}, \ref{B9}
  and Lemma \ref{B5} \textit{(ii)}, we conclude
  \begin{equation*}
      \beta_{\theta}=\beta_1+o_{\theta}(1)< \tilde{J}_{d}+\frac{\2-1}{2\2}S_{HL}^{\frac{\2}{\2-1}}=m_{d,\theta}+\left(\frac{\2-1}{\2} \right)\left(\frac{S_{HL}^{\2}}{\theta} \right) ^{\frac{1}{\2-1}}+o(1).
  \end{equation*}
It gives us that for $\theta$ is close to $1^-$ for all $\theta\in (1-\e,1)$ ($\e$ small enough), we have 
\begin{equation}\label{b36}
    \beta_{\theta}+\eta<m_{d,\theta}+\left(\frac{\2-1}{\2} \right)\left(\frac{S_{HL}^{\2}}{\theta} \right) ^{\frac{1}{\2-1}},
\end{equation}
  for some $\eta>0 $ independent of $\theta$. Resuming the information from \eqref{b40}, \eqref{b36} and the fact that $\lim\limits_{n\to\infty}\tilde{J}_{d,\theta}(u_n)=\beta_{\theta},$ we deduce  a self-contradictory inequality
  $$m_{d,\theta}+\left(\frac{\2-1}{\2} \right)\left(\frac{S_{HL}^{\2}}{\theta} \right) ^{\frac{1}{\2-1}}\le\beta_{\theta} <\beta_{\theta}+\eta<m_{d,\theta}+\left(\frac{\2-1}{\2} \right)\left(\frac{S_{HL}^{\2}}{\theta} \right) ^{\frac{1}{\2-1}}.$$
  Thus, $l=0$, which proves our claim. Therefore, $u_n\to u_d$ strongly in $X_0$. Hence, $u_{\theta}$ is a critical point of $\tilde{J}_{d,\theta}$ 
on the constrained $S_d^+$ at level $\beta_{\theta}$.\par 
  By Lagrange multiplier principle, we know that $u_{\theta}$ satisfies
  \begin{equation*}
  (-\Delta)^s u_{\theta} = \lambda_{\theta} u^+_{\theta} +\alpha|u^+_{\theta}|^{p-2}u^+_{\theta}+ \theta\left( \int\limits_{\Omega} \frac{|u^+_{\theta}(y)|^{2^{*}_{\mu ,s}}}{|x-y|^ \mu}\, \dy\right)  |u^+_{\theta}|^{2^{*}_{\mu ,s}-2}u^+_{\theta}\;~ \text{in} \; \Omega,
  \end{equation*} 
  for $\alpha\le0$, and
  
   \begin{equation*}
  (-\Delta)^s u_{\theta} = \lambda_{\theta} u^+_{\theta} +\alpha\theta|u^+_{\theta}|^{p-2}u^+_{\theta}+ \theta\left( \int\limits_{\Omega} \frac{|u^+_{\theta}(y)|^{2^{*}_{\mu ,s}}}{|x-y|^ \mu}\, \dy\right)  |u^+_{\theta}|^{2^{*}_{\mu ,s}-2}u^+_{\theta}\;~ \text{in} \; \Omega,
  \end{equation*} 
  for $\alpha>0.$ Next to show that $u_{\theta} $ is positive, multiplying above equations by $u_{\theta}$ and integrating on $\Omega$, we obtain that 
  $$\|u_{\theta}^-\|_{X_0}^2=0,$$ 
  which implies that $u_{\theta}\ge0$. By using maximum principle \cite{giacomoni2020regularity, cheng2017maximum}, we have $u_{\theta}>0$ and solves \eqref{b13} for $\alpha\le0$ (\eqref{b14} for $\alpha>0$). 
\QED
\end{proof} 
\begin{prop}\label{B3}
	Let $d$ satisfies \eqref{M1},\eqref{M2},\eqref{M3} depending on $\alpha$ and $p$. Then there exists a sequence $\{u_n\}\subset S_d^+$ which is bounded in $X_0$ such that $\lim\limits_{n\to\infty}\tilde{J}_d(u_n)=\beta(d)$ and $(\tilde{J}_d|_{S_d})'(u_n)\to 0$ as $n\to\infty$. 
\end{prop}
\begin{proof}
	Using Proposition \ref{B7}, we can take $\theta_n\to 1^-$ and $u_n=u_{\theta_n}$, which solves \eqref{b13} or  \eqref{b14} depending  on $\al$. We aim to show that $\{u_n\}$ is bounded in $X_0$. Using  Lemma \ref{B7} \textit{(ii)}  for  $n$ large enough such that $\theta_n$ is close to $1^-$ then, we have $\beta_{\theta_n}\le 2\beta_{1}$.\par 
	For $\alpha\le 0$, $u_{\theta_n}\in \mathcal{T}_{\theta_n}$, we have for large $n$ such that $p\delta_p<2\2\theta_n$,
	\begin{align*}
	2\beta_{1}\ge \beta_{\theta_n}=\tilde{J}_{d,\theta_n}(u_n)&>\left(\frac{1}{2}-\frac{1}{2\2\theta_n} \right)\|u_n\|^2_{X_0}+\alpha\left(\frac{\delta_p}{2\2\theta_n}-\frac{1}{p}\right)\|u_n^+\|_p^p\\
	&\ge \left(\frac{1}{2}-\frac{1}{2\2\theta_n} \right)\|u_n\|_{X_0}^2,
	\end{align*}
	which shows boundedness of $\{u_n\}$ in $X_0$. For $\alpha>0$ and $2+\frac{4s}{N}<p<2^*_s$, again using the fact that $u_{\theta_n}\in \mathcal{T}_{\theta_n}$, we have for large $n$ such that $p\delta_p\theta_n>2$
	\begin{equation*}
	2\beta_{1}\ge\beta_{\theta_n}=\tilde{J}_{d,\theta_n}(u_n)>\left(\frac{1}{2}-\frac{1}{p\delta_p\theta_n}\right)\|u_n\|^2_{X_0},
	\end{equation*}
	which shows boundedness of $\{u_n\}$ in $X_0$.\par 
	As $\{u_n\}$ is bounded $X_0$ and $\tilde{J}_{d,\theta_n}=\beta_{\theta_n}$, by Lemma \ref{B5} \textit{(ii)}, we get
	\begin{equation*}
	\lim\limits_{n\to\infty}\tilde{J}_d(u_n)=\lim\limits_{n\to\infty}\tilde{J}_{d,\theta_n}(u_n)=\lim\limits_{n\to\infty}\beta_{\theta_n}=\beta_{1}
	\end{equation*}
	As $\beta_{1}=\beta(d)$ and $\{u_n\}$ is $X_0$ bounded, so we have 
	\begin{equation*}
	(\tilde{J}_d|_{S_d^+})'(u_n)=(\tilde{J}_{d,\theta_n}|_{S_d^+})'(u_n)+o_n(1)\to0~~\text{as}~n\to\infty, 
	\end{equation*}
	which completes the proof.\QED
\end{proof}
\maketitle
\section{Existence of MP-Type solution}\label{Sa4}
In this section, we will prove the existence of a second solution using the uniform mountain pass theorem. We use the minimizer of the inequality (Proposition \ref{hls}) to conclude that the Palais-Smale condition holds at $\beta(d)$.
Let $\zeta\in C_0^{\infty}$ be a radial function, such that
\begin{equation*}
\zeta(x)\equiv
    \begin{cases}
        1~~~~&\text{for}~|x|\le~ R,\\
        0 &\text{for}~|x|\ge ~2R,
    \end{cases}
\end{equation*}
here $B_{2R} \subset \Omega$. Let $v_{\epsilon}=\zeta U_{\epsilon}, $ where $U_{\e}$ is define in \eqref{b35}.

\begin{lemma}\label{B23} Let $s\in(0,1)$ and $N>2s$. Then, as $\epsilon\to 0$, we have
	\begin{itemize}
		\item[(i)]
		 \begin{align*}
	\|v_{\epsilon}\|_{X_0}^2&\le S_{HL}^{\frac{\2}{\2-1}}.
    \end{align*}
	\item[(ii)]\begin{equation*}
	\|v_{\epsilon}\|_2^2=\begin{cases}
	C_s\epsilon^{2s}+O(\epsilon^{N-2s})&\text{if}~n>4s,\\
C_s\epsilon^{2s}|log\epsilon|+O(\epsilon^{2s})&\text{if}~n=4s,\\
	C_s\epsilon^{N-2s}+O(\epsilon^{2s})&\text{if}~n<4s.\\
	\end{cases}
	\end{equation*}
	where $C_s$ is some positive constant, depending on $s$.
    \item[(iii)] \begin{equation*}
        \|v_{\epsilon}\|_p^p=\begin{cases}
            O(\epsilon^{N-\frac{N-2s}{2}p}) &\text{if}~N>\frac{p}{p-1}2s,\\
            O(\epsilon^{\frac{N}{2}}|\log\epsilon|)&\text{if}~N=\frac{p}{p-1}2s,\\
            O(\epsilon^{\frac{N-2s}{2}p})&\text{if}~N<\frac{p}{p-1}2s.
        \end{cases}
    \end{equation*}
    \item[(iv)] \begin{equation*}
        \|v_{\e}\|_{NL}^{2\2}\le S_{HL}^{\frac{2\2}{\2-2}}+O(\epsilon^N).
    \end{equation*}
    \begin{equation*}
        \|v_{\e}\|_{NL}^{2\2}\ge S_{HL}^{\frac{2\2}{\2-2}}-O(\epsilon^N).
    \end{equation*}
		\item[(v)]\begin{align*}
\|u_d+kv_{\epsilon}\|_{NL}^{2\2}\ge\|u_d\|_{NL}^{2\2}+\|kv_{\epsilon}\|_{NL}^{2\2}+\hat{A}k^{2\2-1}\int_{\Omega}\int_{\Omega}\frac{(v_{\epsilon}(x))^{\2}(v_{\epsilon}(y))^{\2-1}u_d}{|x-y|^{\mu}}\dx\dy\\+2\2k\int_{\Omega}\int_{\Omega}\frac{(u_{d}(x))^{\2}(u_d(y))^{\2-1}v_{\epsilon}(x)}{|x-y|^{\mu}}\dx\dy-o(\epsilon^{\frac{N-2s}{2}}).
        \end{align*}
\end{itemize}
\end{lemma}
\begin{proof}
    For \textit{(i) - (iv)} refer \cite{mukherjee2016fractional,servadei2015brezis,he2022normalized}.\par
    \textit{(v)} If $\2>2$, we can use \cite{goel2019kirchhoff}, if $1<\2\le2$ then we can apply the same assertion as in \cite{goel2019kirchhoff} and  the fact that $v_{\epsilon}$ has a compact support.
\end{proof}
\begin{lemma}\label{B22}\cite[Lemma 5.5]{song2024two}
    For any $\phi\in X_0$, we have
    \begin{itemize}
    \item[(i)] $\langle v_{\epsilon},\phi\rangle_{X_0}\le {B}_1R^2\epsilon^{\frac{N-2s}{2}}$,
    \item[(ii)] $\displaystyle\int_{\Omega}\int_{\Omega}\frac{(v_{\epsilon}(x))^{\2}(v_{\epsilon}(y))^{\2-1}\phi(y)}{|x-y|^{\mu}}\dx\dy\ge B_2\epsilon^{\frac{N-2s}{2}} .$
\end{itemize}
\end{lemma}
	\begin{prop}\label{B9}
	Considering the same assumption as in Theorem \ref{B11}, and for when $\alpha\ne0$, $N\in(2s,6s)$, then
	$$\beta(d)<m_d+\frac{\2-1}{2\2}S_{HL}^{\frac{\2}{\2-1}}.$$
    \end{prop}
	\begin{proof}
		Let $w_{\epsilon,k}=u_d+kv_{\epsilon}$, for $k\ge0$. Then $w_{\epsilon,k}>0$ on $\Om$. Define
		$$W_{\epsilon,k}(x)=b^{\frac{N-2s}{2}}w_{\epsilon,k}(bx) .$$
		For $b=\left(\frac{\|w_{\epsilon,k}\|_2}{\sqrt{d}}\right)^{\frac{1}{s}}\ge 1$, we obtain $W_{\epsilon,k}\in X_0$, and $\|W_{\epsilon,k}\|_2^2=d$. Let $\overline{W}_t=t^{\frac{N}{2}}W_{\epsilon,\tilde{k}}(tx),$ for $t\ge1$, where $\tilde{k}$ is defined below. Set
		\begin{align*}
		\Phi(t)=\tilde{J}_d(\overline{W}_t)=\frac{t^{2s}}{2}\|W_{\epsilon,\tilde{k}}\|_{X_0}^2-\frac{\alpha t^{sp\delta_p}}{p}\|W_{\epsilon,\tilde{k}}\|_p^p-\frac{t^{s2\2}}{2.2^*_{\mu,s}}\|W_{\epsilon,\tilde{k}}\|_{NL}^{2\2},~~t\ge1.
		\end{align*}
		Then, we have
		\begin{align*}
		\Phi'(t)=st^{2s-1}\|W_{\epsilon,\tilde{k}}\|_{X_0}^2-\alpha s\delta_pt^{sp\delta_p-1}\|W_{\epsilon,\tilde{k}}\|_p^p-st^{2\2-1}\|W_{\epsilon,\tilde{k}}\|^{22^*_{\mu,s}}_{NL}.
		\end{align*} 
		By Lemma \ref{B23}, we derive
		\begin{align*}
\|W_{\epsilon,\tilde{k}}\|_{X_0}^2&=\|u_d\|_{X_0}^2+\tilde{k}^2S_{HL}^{\frac{\2}{\2-1}}+o_{\epsilon}(1),\\
		\|W_{\epsilon,\tilde{k}}\|_{NL}^{2\2}&\ge \|u_d\|_{NL}^{2\2}+\tilde{k}^{2\2}S_{HL}^{\frac{\2}{\2-1}}+o_{\epsilon}(1),\\
\|W_{\epsilon,\tilde{k}}\|^p_p&=\|u_d\|_p^p+o_{\epsilon}(1).		
		\end{align*}		
		We can choose $\tilde{k}$ large such that $\Phi'(t)<0$, for all $t>1$. Consequently,  $\tilde{J}_d(\overline{W}_t)\le \tilde{J}_d(W_{\epsilon,\tilde{k}}).$ Observe that $\Phi(t)\to-\infty$ as $t\to\infty$.
Since 	
	\begin{align*}
		\tilde{J}_d(W_{\epsilon,k})&=\frac{1}{2}\|W_{\epsilon,k}\|^2_{X_0}-\frac{1}{2\2}\|W_{\epsilon,k}\|_{N.L}^{2\2}-\frac{\alpha}{p}\|W_{\epsilon,k}\|_p^p\\
		&\le \frac{1}{2}\|W_{\epsilon,k}\|^2_{X_0}+\frac{|\alpha|}{p}\|W_{\epsilon,k}\|_p^p,
	\end{align*}
	so we can find a $k_0\in (0,\infty)$, such that
	\begin{equation}\label{b21}
	\tilde{J}_d(W_{\epsilon,k})\le \tilde{J}_d(u_d)+\frac{\2-1}{2\2}(S_{HL})^{\frac{\2}{\2-1}},
	\end{equation}
 for all $k\in(0,k_0]$.  Resuming all the above information, we get 
	\begin{equation*}
		\begin{aligned}
\tilde{J}_d(W_{\epsilon,k})&=\tilde{J}_d(w_{\epsilon,k})+\frac{\alpha}{p}(1-b^{p(\delta_p-1)})\|w_{\epsilon,k}\|_p^p\\
	=&\tilde{J}_d(u_d)+\frac{k^2}{2}\|v_{\epsilon}\|_{X_0}^2+k\langle u_d,v_{\epsilon}\rangle_{X_0}+\frac{1}{2\2}\left(\|u_d\|_{NL}^{2\2}-\|u_d+kv_{\epsilon}\|_{NL}^{2\2}\right)\\
		&+\frac{\alpha}{p}\left(\|u_d\|_p^p-\|u_d+kv_{\epsilon}\|_p^p \right)+\frac{\alpha}{p}(1-b^{p(\delta_p-1)})\|u_d+kv_{\epsilon}\|_p^p.
        \end{aligned}
        \end{equation*}
        With the help of Lemma \ref{B23} \textit{(v)}, one has 
        \begin{align}\label{b30}
   \notag \tilde{J}_d(W_{\epsilon,k})        
		\le &\tilde{J}_d(u_d)+\frac{k^2}{2}\|v_{\epsilon}\|_{X_0}^2+k\langle u_d,v_{\epsilon}\rangle_{X_0}-\frac{\alpha}{p}\left(\|kv_{\epsilon}\|_p^p+kp\int_{\Omega}u_d^{p-1}v_{\epsilon}+\hat{A}_1k^{p-1}\int_{\Omega}v_{\epsilon}^{p-1}u_d \right)\\
		& \notag -\frac{1}{2\2}\left(\|kv_{\epsilon}\|_{NL}^{2\2}+2\2k\int_{\Omega}\int_{\Omega}\frac{|u_{d}(x)|^{\2}|u_d(y)|^{\2-1}v_{\epsilon}(x)}{|x-y|^{\mu}}\dx\dy,\right.\\&\left.\qquad+\hat{A}_2k^{2\2-1}\int_{\Omega}\int_{\Omega}\frac{|v_{\epsilon}(x)|^{\2}|v_{\epsilon}(y)|^{\2-1}u_d(y)}{|x-y|^{\mu}}\dx\dy \right)+\frac{\alpha}{p}(1-b^{p(\delta_p-1)})\|u_d+kv_{\epsilon}\|_p^p.
	\end{align}
Since $u_d$ is a weak solution of \eqref{F}, we conclude
	$$ \langle u_d,v_{\epsilon}\rangle_{X_0}=\lambda_d\langle u_d,v_{\epsilon} \rangle+\int_{\Omega}\int_{\Omega}\frac{|u_d(x)|^{\2}|u_d(y)|^{\2-1}v_{\epsilon}}{|x-y|^{\mu}}\dx\dy+\alpha\int_{\Omega}|u_d|^{p-1}v_{\epsilon},$$
	using this in \eqref{b30}, we obtain
	\begin{align}\label{b23}
	\notag \tilde{J}_d(W_{\epsilon,k})&\le \tilde{J}_d(u_d)+\frac{k^2}{2}\|v_{\epsilon}\|_2^2+\lambda k\langle u_d,v_{\epsilon}\rangle-\frac{\alpha}{p}\left(\|kv_{\epsilon}\|_p^p+\hat{A}_1\int_{\Omega}v_{\epsilon}^{p-1}u_d \right)+\frac{\alpha}{p}(1-b^{p(\delta_p-1)})\|u_d+v_{\epsilon}\|_{p}^p\\&\quad  -\frac{1}{2\2}\|kv_{\epsilon}\|_{NL}^{2\2} -\frac{\hat{A}_2k^{2\2-1}}{2\2}\int_{\Omega}\int_{\Omega}\frac{|v_{\epsilon}(x)|^{\2}|v_{\epsilon}(y)|^{\2-1}u_d}{|x-y|^{\mu}}\dx\dy.	
	\end{align}
	\textbf{Case 1:} $\alpha=0$\par
    In this case, \eqref{b23} can be expressed as:
	\begin{align*}
	\tilde{J}_d(W_{\epsilon,k})&\le \tilde{J}_d(u_d)+\frac{k^2}{2}\|v_{\epsilon}\|_2^2+\lambda k\langle u_d,v_{\epsilon}\rangle-\frac{\hat{A}_2k^{2\2-1}}{2\2}\int_{\Omega}\int_{\Omega}\frac{|v_{\epsilon}(x)|^{\2}|v_{\epsilon}(y)|^{\2-1}u_d}{|x-y|^{\mu}}\dx\dy\\
    &\quad-\frac{1}{2\2}\|kv_{\epsilon}\|_{NL}^{2\2}\\
	&\le \tilde{J}_d(u_d)+\frac{\2-1}{2\2}(S_{HL})^{\frac{\2}{\2-1}}+k\lambda_d\hat{B}_1R^2\epsilon^{\frac{N-2s}{2}}-k^{2\2-1}\hat{B}_2\epsilon^{\frac{N-2s}{2}},
	\end{align*}
	 by using Lemma \ref{B22}, for some large $k_1$ we have $(k_1\lambda_d\hat{B}_1R^2\epsilon^{\frac{N-2s}{2}}-k_1^{2\2-1}\hat{B}_2\epsilon^{\frac{N-2s}{2}})<0$, and we obtain \eqref{b21} for $k\in [k_1,\infty).$ When $k\in(k_0,k_1)$, we can choose $R$ small enough so that \eqref{b21} holds.\par
	\textbf{Case 2:} $\alpha<0$\par
	Since $\delta_p<1$ and $b>1$, it follows that $(1-b^{p(\delta_p-1)})>0$. Therefore, for $\alpha<0$ and $2s<N<6s$, we find that
	  \begin{align*}
	  \tilde{J}_d(W_{\epsilon,k})&\le \tilde{J}_d(u_d)+k^2\|v_{\epsilon}\|_{X_0}^2+\lambda_dk\langle u_d,v_{\epsilon}\rangle-\frac{1}{2\2}\|kv_{\epsilon}\|_{NL}^{2\2}+\frac{|\alpha|}{p}(\|kv_{\epsilon}\|_p^p+\hat{A_1}k^{p-1}\int_{\Omega}v_{\epsilon}^{p-1}u_d)\\
	  &\quad-\frac{\hat{A}k^{2\2-1}}{2\2}\int_{\Omega}\int_{\Omega}\frac{|v_{\epsilon}(x)|^{\2}|v_{\epsilon}(y)|^{\2-1}u_d}{|x-y|^{\mu}}\dx\dy\\
	  &<\tilde{J}_d(u_d)+\frac{\2-1}{2\2}(S_{HL})^{\frac{\2}{\2-1}}+k\lambda_d\hat{B}_3R^2\epsilon^{\frac{N-2s}{2}}-k^{2\2-1}\hat{B}_4\epsilon^{\frac{N-2s}{2}},
	  \end{align*}
      following the same reasoning as for $\alpha=0$, we can choose $k_2$ large enough so that \eqref{b21} holds and for $k\in (k_0,k_1)$, we can also find $R$ small enough, so that \eqref{b21} holds.\par
	  \textbf{Case 3:} $\alpha>0$ \par
	 Since, $b^2=\frac{\|u_d+kv_{\epsilon}\|^{\frac{2}{s}}_2}{d}$ therefore, we have
	  $$1-b^{p(\delta_p-1)}<p(1-\delta_p)\left(\frac{k}{d}\langle u_d,v_{\epsilon}\rangle+\frac{k^2}{2d}\|v_{\epsilon}\|_2^2\right).$$
	 In implies that for   for $\alpha>0$ and $N\in (2s,6s)$ we conclude that
	  \begin{align*}
	  \tilde{J}_d(W_{\epsilon,k})&\le \tilde{J}_d(u_d)+k^2\|v_{\epsilon}\|_{X_0}^2+\lambda_dk\langle u_d,v_{\epsilon}\rangle-\frac{k^{2\2}}{2\2}\|v_{\epsilon}\|_{NL}^{2\2}+\frac{\alpha}{p}(1-b^{p(\delta_p-1)})\|w_{\epsilon,k}|_p^p\\&\quad-\frac{\hat{A}k^{2\2-1}}{2\2}\int_{\Omega}\int_{\Omega}\frac{|v_{\epsilon}(x)|^{\2}|v_{\epsilon}(y)|^{\2-1}u_d}{|x-y|^{\mu}}\dx\dy\\
	  &\le \tilde{J}_d(u_d)+k^2\|v_{\epsilon}\|_{X_0}^2+\lambda_dk\langle u_d,v_{\epsilon}\rangle+{\alpha(1-\delta_p)}\left(\frac{k}{d}\langle u_d,v_{\epsilon}\rangle+\frac{k^2}{2d}\|v_{\epsilon}\|_2^2\right)\|w_{\epsilon,k}\|_p^p\\&\quad-\frac{k^{2\2}}{2\2}\|v_{\epsilon}\|_{NL}^{2\2}-\frac{\hat{A}k^{2\2-1}}{2\2}\int_{\Omega}\int_{\Omega}\frac{|v_{\epsilon}(x)|^{\2}|v_{\epsilon}(y)|^{\2-1}u_d}{|x-y|^{\mu}}\dx\dy\\
	  &<\tilde{J}_d(u_d)+\frac{\2-1}{2\2}(S_{HL})^{\frac{\2}{\2-1}}+k\lambda_d\hat{B}_5R^2\epsilon^{\frac{N-2s}{2}}-k^{2\2-1}\hat{B}_5\epsilon^{\frac{N-2s}{2}}. 
	  \end{align*}
	Now following the assertions as in the above two cases, we can easily find some large $k_3$ and choose $R$ too small that again \eqref{b21} holds for $k\in (k_0,k_3)$.\par 
	  With this, we are in a situation to construct a suitable path on $S_d^+$.  Define $\gamma(\ell):=W_{\epsilon,2\ell\tilde{k}}$ for $\ell\in [0,1/2]$ and $\gamma(\ell):=\overline{W}_{2(\ell-1/2)t_0+1}$ for $\ell\in[1/2,1]$, where $t_0$ is large enough such that $\tilde{J}_d(\overline{W}_{t_0+1})<\tilde{J}_d(u_d)$ and we have $\overline{W}_{t_0+1}\not\in\mathcal{T}$. Therefore $\gamma\in \Gamma$ and $\displaystyle\sup_{l\in[0,1]} \tilde{J}_d(\gamma)<\tilde{J}_d(u_d)+\frac{\2-1}{2\2}(S_{HL})^{\frac{\2}{\2-1}}$, implies that 
	  $$\beta(d)<\tilde{J}_d(u_d)+\frac{\2-1}{2\2}(S_{HL})^{\frac{\2}{\2-1}}=m_d+\frac{\2-1}{2\2}(S_{HL})^{\frac{\2}{\2-1}},$$
	  which completes the proof.\QED 
	\end{proof}

\textbf{Proof of theorem 1.2} From Proposition \ref{B3},  we get a bounded sequence $\{u_n\}\subset S_d^+$, such that $\lim\limits_{n\to\infty}\tilde{J}_d(u_n)=\beta(d)$ and $(\tilde{J}_d|_{S_d^+})'(u_n)\to 0$ as $n\to\infty.$  We can assume that up to a subsequence, 
\begin{align*}
u_n&\rightharpoonup\hat{u}_d~~~ \text{in}~X_0,\\
(|x|^{-\mu}*|u_n|^{\2})|u_n|^{\2}&\rightharpoonup (|x|^{-\mu}*|\hat{u}_d|^{\2})|\hat{u}_d|^{\2}~~\text{weakly}
~\text{in}~ L^{\frac{2N}{N+2s}}(\Om),\\
u_n&\to \hat{u}_d~~~\text{in}~L^r(\Omega)~\text{for}~2\le r<2^*.\\
u_n&\to \hat{u}_d~~~\text{a.e. in }~\Omega.
\end{align*}	
        Assume $\omega_n=u_n-\hat{u}_d$ then using the fact that 
		$(\tilde{J}_d|_{S_d^+})'(u_n)\to 0$,  there exists $\lambda_n$ such that $\tilde{J}_d'(u_n)-\lambda_nu_n^+\to 0.$ Let $\hat{\lambda}_d$ be the Lagrange multiplier corresponding to $\hat{u}_d$. Using the same arguments as in the proof of Theorem \ref{B11}, we have
		\begin{equation*}
		\|\omega_n\|^2_{X_0}=\|\omega_n\|_{NL}^{2\2}+o_n(1),
		\end{equation*}
	Assume that $\|\omega_n\|^2_{X_0}\to l\ge 0$, then using the definition of $S_{H,L}$ we obtain that	
	\begin{equation*}
	\|\omega_n\|_{X_0}^2\ge S_{HL}\|\omega_n\|_{NL}^2,
	\end{equation*}
	which implies that $l\ge S_{HL}l^{\frac{2}{\2}}$.	We claim that 
		 that $l=0$. Suppose if possible $l>0$, then $l\ge S_{HL}^{\frac{\2}{\2-1}}$, which implies
		$$\tilde{J}_{d}(w_n)\ge \left(\frac{\2-1}{\2} \right)S_{HL}^{\frac{\2}{\2-1}}+o_n(1).$$
		By using the Brezis-Lieb Lemma, we have
		\begin{align}\label{b15}
		\notag \tilde{J}_{d}(u_n)&= \tilde{J}_{d}(u_d)+ \tilde{J}_{d}(w_n)+o_n(1)\\
		&\ge m_{d}+\left(\frac{\2-1}{\2} \right)S_{HL}^{\frac{\2}{\2-1}} +o_n(1).
		\end{align}
		Here we use the fact that $m_d$ is the energy level corresponding to the solution $u_d$.  On the other hand, by Proposition \ref{B9}, we have
		\begin{equation}\label{b16}
		\tilde{J}_d(u_n)+o_n(1)=\beta(d)<m_d+\left(\frac{\2-1}{\2} \right)S_{HL}^{\frac{\2}{\2-1}}+o_n(1),
		\end{equation}
	Utilizing 	 \eqref{b15} and \eqref{b16}, we  get $l=0$.  That is,  $u_n\to \hat{u}_{d}$ strongly in $X_0.$ Now using same method  as in    Proposition \ref{B7}, we get the desired result. \qed

\noindent{\bf Acknowledgements :} 
The first author would like to thank the Science and Engineering Research Board, Department of Science and Technology, Government of India, for the financial support under the grant  SRG/2022/001946. 
\bibliographystyle{plain}
\bibliography{references}

\end{document}